\pdfoutput=1
\documentclass[11pt]{amsart}
\usepackage{latexsym,amsmath,amsfonts,amscd,amssymb}
\usepackage{stmaryrd}
\usepackage{hyperref}
\usepackage{array}   
\usepackage[mathscr]{eucal} 
\usepackage{mathrsfs}
\usepackage{graphicx}
\usepackage{nicematrix}
\usepackage{xypic}
\usepackage{calligra}
\usepackage[T1]{fontenc}
\DeclareFontFamily{OT1}{pzc}{}
\DeclareFontShape{OT1}{pzc}{m}{it}{<-> s * [1.10] pzcmi7t}{}
\DeclareMathAlphabet{\mathpzc}{OT1}{pzc}{m}{it}
\DeclareFontFamily{OT1}{rsfs}{}
\DeclareFontShape{OT1}{rsfs}{n}{it}{<->rsfs10}{}
\DeclareMathAlphabet{\curly}{OT1}{rsfs}{n}{it}
\theoremstyle{plain}
\newtheorem{theorem}{Theorem}[section]
\newtheorem*{theorem*}{Theorem}

\newtheorem{corollary}[theorem]{Corollary}
\newtheorem{lemma}[theorem]{Lemma}
\newtheorem{proposition}[theorem]{Proposition}

\theoremstyle{definition}
\newtheorem{definition}[theorem]{Definition}
\theoremstyle{remark}
\newtheorem{example}[theorem]{Example}
\newtheorem{remark}[theorem]{Remark}

\newtheorem*{claim*}{Claim}
\numberwithin{equation}{section}

\renewcommand{\le}{\leqslant}

\renewcommand{\ge}{\geqslant}
\renewcommand{\setminus}{\smallsetminus}

\newcommand{\R}{\mathbb{R}}

\newcommand{\Z}{\mathbb{Z}}
\newcommand{\C}{\mathbb{C}}

\DeclareMathOperator*{\heightr}{ht}

\newcommand{\gen}[1]{\left< #1 \right>}

\newcommand{\PGL}{\mathrm{PGL}}
\newcommand{\PSL}{\mathrm{PSL}}

\DeclareMathOperator{\ord}{ord}

\DeclareMathOperator{\Spec}{Spec}
\DeclareMathOperator{\ad}{ad}
\DeclareMathOperator{\Ad}{Ad}

\DeclareMathOperator{\rank}{rank}

\DeclareMathOperator{\Hom}{Hom}

\DeclareMathOperator{\Id}{Id}

\DeclareMathOperator{\Sym}{Sym}

\DeclareMathOperator{\Aut}{Aut}

\DeclareMathOperator{\conj}{Conj}

\DeclareMathOperator*{\GL}{GL}
\DeclareMathOperator*{\SL}{SL}
\DeclareMathOperator*{\SO}{SO}

\DeclareMathOperator*{\Sp}{Sp}

\newcommand{\Gr}{\operatorname{Gr}}
\renewcommand{\phi}{\varphi}
\newcommand{\git}{\mathbin{/\mkern-6mu/}}

\newcommand{\liep}{\mathfrak{p}}
\newcommand{\liet}{\mathfrak{t}}

\newcommand{\lieg}{\mathfrak{g}}

\newcommand{\liek}{\mathfrak{k}}
\newcommand{\liel}{\mathfrak{l}}

\newcommand{\liesl}{\mathfrak{sl}}

\renewcommand{\phi}{\varphi}

\usepackage{tikz-cd}
\usetikzlibrary{calc}
\tikzset{curve/.style={settings={#1},to path={(\tikztostart)
			.. controls ($(\tikztostart)!\pv{pos}!(\tikztotarget)!\pv{height}!270:(\tikztotarget)$)
			and ($(\tikztostart)!1-\pv{pos}!(\tikztotarget)!\pv{height}!270:(\tikztotarget)$)
			.. (\tikztotarget)\tikztonodes}},
	settings/.code={\tikzset{quiver/.cd,#1}
		\def\pv##1{\pgfkeysvalueof{/tikz/quiver/##1}}},
	quiver/.cd,pos/.initial=0.35,height/.initial=0}
\tikzset{tail reversed/.code={\pgfsetarrowsstart{tikzcd to}}}
\tikzset{2tail/.code={\pgfsetarrowsstart{Implies[reversed]}}}
\tikzset{2tail reversed/.code={\pgfsetarrowsstart{Implies}}}
\tikzset{no body/.style={/tikz/dash pattern=on 0 off 1mm}}

\begin{document}
	
	\title[Very stable regular $G$-Higgs bundles]{Very stable regular $G$-Higgs bundles}

	\author[Miguel González]{Miguel González}
	\address{Instituto de Ciencias Matem\'aticas \\
		CSIC-UAM-UC3M-UCM \\ Nicol\'as Cabrera, 13--15 \\ 28049 Madrid \\ Spain}
	\email{miguel.gonzalez@icmat.es}
	
	\noindent
	\thanks{
		\noindent
		The project that gave rise to these results received the support of a fellowship from ``la Caixa'' Foundation (ID 100010434). The fellowship code is LCF/BQ/DR23/12000030. The work was also partially supported by Austrian Science Fund (FWF) grant ``Geometry of the tip of the global nilpotent cone'' no. 10.55776/P35847 and by Spanish grant CEX2023-001347-S funded by MICIU/AEI/10.13039/501100011033.
	}

	
	\begin{abstract}
	We give a classification of very stable $G$-Higgs bundles in the generically regular Higgs field case for $G$ an arbitrary connected semisimple complex group. This extends the classification for $G=\GL_n(\C)$ and fixed point type $(1,1,\dots,1)$ given by Hausel and Hitchin.
	\end{abstract}
	
	\maketitle
	
	\section{Introduction}
	
	Given a smooth projective complex curve $C$ of genus $g \ge 2$ with canonical bundle $K_C$ and a semisimple complex Lie group $G$ with Lie algebra $\lieg$, a $G$-Higgs bundle is a pair $(E,\varphi)$ consisting of a principal $G$-bundle $E$ over $C$ together with a section $\varphi$ of the vector bundle $E(\lieg) \otimes K_C$. $G$-Higgs bundles were introduced by Hitchin \cite{hitchin_self-duality_1987, hitchin_stable_1987} motivated by the study of the Yang--Mills self-duality equations and their dimensional reduction to a Riemann surface.
	
	After introducing appropriate stability notions, moduli spaces $\mathcal M(G)$ of isomorphism classes of polystable $G$-Higgs bundles can be constructed \cite{hitchin_self-duality_1987,nitsure_moduli_1991}. They have been intensely studied in the literature due to their many interesting geometric properties. For example, they are homeomorphic to the character varieties $\Hom(\pi_1(C),G)\git G$ and their smooth locus carries a hyperkähler structure. Furthermore, there is a completely integrable system \cite{hitchin_stable_1987} with respect to one of its complex symplectic structures
	$$h_G : \mathcal M(G) \to \mathcal A_G$$
	\noindent onto an affine space $\mathcal A_G$, known as the \textit{Hitchin system}. 
	
	These moduli spaces also play an important role in relation to Langlands duality and mirror symmetry. Indeed, denoting by $G^\vee$ the Langlands dual group of $G$, there is an identification $\mathcal A_G \simeq \mathcal A_{G^\vee}$ via which $h_G$ and $h_{G^\vee}$ become dual Lagrangian fibrations \cite{donagi_langlands_2012}. A phenomenon predicted by mirror symmetry \cite{kapustin_electric-magnetic_2007} is that complex Lagrangian subvarieties of $\mathcal M(G)$, sometimes referred to as \textit{BAA-branes}, should correspond via this duality to hyperholomorphic vector bundles over hyperholomorphic subvarieties of $\mathcal M(G^\vee)$, or \textit{BBB-branes}.       
	
	Motivated by this, the notion of \textit{very stable Higgs bundle} was introduced by Hausel and Hitchin in \cite{hausel_very_2022}, extending the notion of a very stable vector bundle from \cite{laumon_analogue_1988}. This can be defined in the context of the natural $\C^\times$-action on $\mathcal M(G)$ given by
	$$(E,\varphi) \mapsto (E,\lambda \varphi).$$
	A very stable Higgs bundle is a smooth fixed point $\mathcal E \in \mathcal M(G)^{s\C^\times}$ such that its \textit{upward flow}
	$$W^+_{\mathcal E} := \{(E,\varphi) \in \mathcal M(G) : \lim\limits_{\lambda \to 0}(E,\lambda \varphi) = \mathcal E\}$$
	\noindent does not contain any nilpotent Higgs bundle other than $\mathcal E$. Upward flows are in general Lagrangian subvarieties of $\mathcal M(G)$. For very stable fixed points they are closed and $h_G$ restricts to a proper map on them, rendering the study of their mirror symmetry aspects (e.g. \cite[Section 6]{hausel_very_2022}) more manageable.
	
	Very stable Higgs bundles have been studied in the literature in the past, mainly for $G=\GL_n(\C)$. In order to explain this, recall (\cite[Section 7]{hitchin_self-duality_1987}, \cite[Section 4]{simpson_higgs_1992}, \cite[Lemma 9.2]{hausel_mirror_2003}, \cite[Section 3.1]{hausel_very_2022} and \cite{garcia-prada_y_2013, garcia-prada_motives_2014}) that in this case a fixed point of the $\C^\times$-action can be described as a tuple of $k$ vector bundles $E_0,\dots,E_{k-1}$ over $C$ with ranks adding to $n$, and $k-1$ maps $\varphi_i : E_{i} \to E_{i+1} \otimes K_C$. The original work of Laumon \cite{labourie_cyclic_2017} on very stable vector bundles corresponds to the case $k=1$, where it is shown that the very stable locus is an open dense subset of the fixed point component. In this same case, Pauly and Peón-Nieto \cite{pauly_very_2019} related for the first time very stable bundles with the properness of the Hitchin map on their upward flows. Later, Hausel and Hitchin \cite{hausel_very_2022} treated the case of $k=n$, giving a complete classification of the very stable points in terms of the reducedness of a divisor associated to the maps $\varphi_{i}$ of line bundles. In this same setting, \cite{gonzalez_even_2024} gives the analogous classification when restricted to the fixed point locus $\mathcal M(\GL_n(\C))^{-1}$, a subvariety of interest due to its connection with character varieties for real Lie groups \cite{garcia-prada_involutions_2019}. More recently, the case of $k=2$ has been studied by Peón-Nieto in \cite{peon-nieto_wobbly_2023}, where it is shown that most fixed point components do not have very stable Higgs bundles at all. There has also been work for arbitrary $G$ \cite{zelaci_very_2018} in the case $\varphi = 0$, which results in the natural generalisation of the results of Laumon, Pauly and Peón-Nieto to the principal bundle setting.   
	
	The main goal of this paper is to extend the classification of \cite{hausel_very_2022} in the case of $G=\GL_n(\C)$ and $k=n$ (also referred to as \textit{type} $(1,1,\dots,1)$ due to the vector bundles $E_j$ having rank one) to an arbitrary connected semisimple complex Lie group $G$. The corresponding fixed points can be characterised by the fact that the Higgs field $\varphi$ takes values generically in the orbit of regular nilpotent elements in $\lieg$. 
	
	More precisely, fixed points of the $\C^\times$-action can be classified by considering $\Z$-gradings of the Lie algebra, that is, vector space direct sum decompositions
	$$\lieg = \bigoplus_{j \in \Z}\lieg_j$$
	\noindent such that $[\lieg_i,\lieg_j] \subseteq \lieg_{i+j}$. In this situation there exists an element $\zeta \in \lieg_0$ such that $\lieg_j = \ker(\ad(\zeta)-j\Id)$. Fixed points of the $\C^\times$-action then correspond \cite[Section 4.2]{biquard_arakelov-milnor_2021} (see also \cite{simpson_constructing_1988}) to $G$-Higgs bundles $(E,\varphi)$ with a reduction of structure group $E'$ of $E$ from $G$ to the centraliser $C_G(\zeta)$ and such that $\varphi$ belongs to $E'(\lieg_i) \otimes K_C$ for some $i \neq 0$. Our case of interest is then that of $i=1$ and $\zeta = \sum_{j=1}^r\omega_j^\vee$ for $r=\rank \lieg$ and a choice of maximal torus $T \subseteq G$ and simple roots $\{\alpha_1,\dots,\alpha_r\}$ with their corresponding fundamental coweights $\{\omega_1^\vee,\dots,\omega_r^\vee\}$. We also say that these fixed points are of \textit{Borel type}, due to the fact that the associated parabolic subalgebra $\bigoplus_{j \ge 0} \lieg_j$ is a Borel subalgebra and $C_G(\zeta)=T$ is the maximal torus which is a Levi subgroup for the Borel subgroup corresponding to the given choice of fundamental coweights.
	
	Already in \cite[Section 8.1]{hausel_very_2022} some remarks about this situation were outlined. In particular, one direction of the classification of the very stable $G$-Higgs bundles of Borel type was conjectured \cite[Conjecture 8.2]{hausel_very_2022}. The classification given in Theorem \ref{verystablecharact} of this paper agrees with that conjecture and also proves the converse statement.
	
	In order to understand the classification, note that in the $\Z$-grading for Borel type we have
	$$\lieg_1 = \bigoplus_{i=1}^r\lieg_{\alpha_i}$$
	\noindent and the structure group $C_G(\zeta) = T$ acts on each root space $\lieg_{\alpha_i}$. Therefore, given a fixed point $(E,\varphi) \in \mathcal M(G)^{\C^\times}$ of Borel type and a point $c \in C$ in the base curve, we can look at the order $a_i^c \in \Z_{\ge 0}$ of vanishing at $c$ of the Higgs field along the line subbundle $E'(\lieg_{\alpha_i})\otimes K_C$. This gives a divisor valued in the space of dominant coweights of $\lieg$:
	$$\mu_{(E,\varphi)} := \sum_{c \in C}\sum_{i=1}^ra_i^c\omega_i^\vee c.$$
	There is a natural partial order in the space of dominant coweights, where $\lambda \le \mu$ if and only if $\mu - \lambda$ can be expressed as a sum of positive coroots. Minimal elements with respect to that order are called \textit{minuscule}. The classification result can then be stated in these terms as follows. 
	\begin{theorem*}[Theorem \ref{verystablecharact}]
		Let $(E,\varphi)$ be a smooth $\C^\times$-fixed point of Borel type with associated multiplicity divisor $\mu_{(E,\varphi)}$. Then $(E,\varphi)$ is very stable if and only if all the coefficients in $\mu_{(E,\varphi)}$ are minuscule.
	\end{theorem*} 
	The main tool involved in this classification is that of \textit{Hecke transformations}. After choosing a principal $G$-bundle $E$, a point $c \in C$ and a suitable trivialisation, each point of the affine Grassmannian $\sigma \in \Gr_G = G((z))/G[[z]]$ gives a new principal $G$-bundle $\mathcal H_\sigma(E)$ which is isomorphic to $E$ over $C \setminus \{c\}$. This is explained and studied thoroughly in \cite{wong_hecke_2013}. The technique can be naturally extended in the presence of a Higgs field $\varphi$, where now only points in a certain subspace of $\Gr_G$ can be used. This subspace is an \textit{affine Springer fibre} over the element of $\lieg[[z]]$ induced by $\varphi$ locally around $c$.   
	
	The classification is carried out by starting with an everywhere regular fixed point (the nilpotent Higgs bundle in a Hitchin section as in \cite{hitchin_lie-groups_1992}), which is readily seen to be very stable, and relating it to the other fixed points of Borel type and their upward flows by means of Hecke transformations. This is analogous to the approach taken in \cite{hausel_very_2022} for $G=\GL_n(\C)$ with the difference that, in that case, working with vector bundles simplifies the Hecke transformations and only usual partial Springer fibres in Grassmannians $\Gr(k,n)$ have to be considered. 
		
	Even though the techniques work independently of the group $G$, there are interesting differences depending on $G$ that can be deduced. For example, the case of $G=\PGL_n(\C)$ has the special feature that every fundamental coweight $\omega_i^\vee$ is minuscule. This implies that every fixed point component of Borel type has (an open, dense subset of) very stable fixed points, as was proven in \cite[Corollary 4.19]{hausel_very_2022}. However, for other types this is not the case. For example, as was remarked in \cite[Section 8.1]{hausel_very_2022}, the case of $G=G_2$ does not have minuscule coweights other than $0$. This, together with our classification, implies that there are no very stable $G_2$-Higgs bundles of Borel type other than the everywhere regular ones. A similar analysis can be carried out for other types, which outside of Dynkin type $A$ results in Borel type fixed point components without any very stable Higgs bundles.
	
	The paper is structured as follows. In Section \ref{secbb} we recall the theory of Bia\l ynicki-Birula about $\C^\times$-actions on semiprojective varieties that will be used in the setting of the natural $\C^\times$-action on $\mathcal M(G)$. In Section \ref{sechiggs} we provide the necessary aspects about the moduli space of $G$-Higgs bundles, including the notions of stability and the $\C^\times$-action, and we describe the elements of the Bia\l ynicki-Birula theory (fixed points, upward and downward flows) in that setting. Section \ref{sechecke} focuses on introducing the affine Grassmannian and showing how it can be used to perform Hecke transformations of principal $G$-bundles and $G$-Higgs bundles. In Section \ref{sectheorem} we use Hecke transformations and the geometry of the affine Grassmannian to prove the main result in Theorem \ref{verystablecharact}. We also deduce some consequences of the theorem such as the statement in terms of vector bundles when $G$ is classical. Finally, in Section \ref{secmult} we show how the virtual equivariant multiplicity introduced in \cite[Section 5]{hausel_very_2022} can be computed in our framework, and we show that it agrees with the Dynkin polynomial in the very stable case, a computation that was already outlined in \cite[Section 8.1]{hausel_very_2022}.
	
	\textbf{Acknowledgements.} I would like to thank Oscar García-Prada and Tamás Hausel for their constant guidance and support and for introducing me to the mathematics appearing in this work. I would also like to thank Mischa Elkner, Guillermo Gallego, Nigel Hitchin and Ana Peón-Nieto for helpful discussions.   
	
	\section{Bia\l ynicki-Birula theory}\label{secbb}
	
	In this section we recall \cite{bialynicki-birula_theorems_1973} (see also \cite[Section 2]{hausel_very_2022}) the main concepts about the Bia\l ynicki-Birula theory of $\C^\times$-actions on varieties that will later play a role in the moduli space of $G$-Higgs bundles. The setting is that of \textit{semiprojective varieties}.
	
	\begin{definition}
		A normal complex quasiprojective variety $X$ equipped with a $\C^\times$-action is \textbf{semiprojective} if the fixed point locus $X^{\C^\times}$ is projective and for every $x \in X$ the limit $\lim\limits_{\lambda \to 0}\lambda x$ exists.
	\end{definition}
	
	We obtain a decomposition of such varieties by using the limit points of the $\C^\times$-action.
	
	\begin{definition}
		Let $X$ be a semiprojective variety and $\alpha \in X^{\C^\times}$. The \textbf{upward flow from $\alpha$} is
		$$W_{\alpha}^+ := \{x \in X : \lim\limits_{\lambda \to 0} \lambda x = \alpha\} \subseteq X.$$
		The upward flows define the \textbf{Bia\l ynicki-Birula partition} \[X = \bigsqcup_{\alpha \in X^{\C^\times}}W_{\alpha}^+.\]
		
		Similarly, the \textbf{downward flow from $\alpha$} is
		$$W_{\alpha}^- := \{x \in X : \lim\limits_{\lambda \to \infty} \lambda x = \alpha\} \subseteq X,$$
		and the \textbf{core} of $X$ is defined to be
		\[\mathcal C := \bigsqcup_{\alpha \in X^{\C^\times}}W_{\alpha}^- \subseteq X.\]  
	\end{definition}
	
	Given a smooth fixed point $\alpha \in X^{s\C^\times}$, its upward and downward flows can be described by studying the tangent space $T_\alpha X$. Since $\alpha$ is fixed, $T_\alpha X$ has an induced $\C^\times$-action which provides a weight space decomposition $T_\alpha X = \bigoplus_{k \in Z}(T_\alpha X)_k$. Let $T_\alpha X^+ := \bigoplus_{k>0}(T_\alpha X)_k$ be the subspace of positive weights and $T_\alpha X^- := \bigoplus_{k<0}(T_\alpha X)_k$ the subspace of negative weights. 
	
	\begin{proposition}[{\cite{bialynicki-birula_theorems_1973}, \cite[Proposition 2.1]{hausel_very_2022}}]\label{bbflows}
		Let $X$ be a semiprojective variety and $\alpha \in X^{s\C^\times}$ a smooth fixed point. The upward flow $W_\alpha^+$ (resp. the downward flow $W_\alpha^-$) is a locally closed $\C^\times$-invariant subvariety of $X$ isomorphic to $T_\alpha^+X$ (resp. $T_\alpha^-X$) as $\C^\times$-varieties.
	\end{proposition}
	
	We now recall the definition \cite[Definition 2.12]{hausel_very_2022} of a very stable point in a semiprojective variety.
	
	\begin{definition}
		A smooth fixed point $\alpha \in X^{s\C^\times}$ is \textbf{very stable} if $W_\alpha^+ \cap \mathcal C = \{\alpha\}$.
	\end{definition}
	
	We will use the following equivalent characterisation given in \cite[Proposition 2.14]{hausel_very_2022}.
	
	\begin{proposition}\label{verystableclosed}
		The point $\alpha \in X^{s\C^\times}$ is very stable if and only if $W_\alpha^+ \subseteq X$ is closed.
	\end{proposition}
	
	Lastly, we consider semiprojective varieties with the additional structure of a symplectic form $\omega \in \Omega^2(X^s)$ of weight one, that is, such that $\Phi_{\lambda}^*\omega = \lambda\omega$ for $\lambda \in \C^\times$, where $\Phi_{\lambda}: X \to X$ is given by the $\C^\times$-action on $X$. In this case, the upward flows are Lagrangian subvarieties.
	
	\begin{proposition}[{\cite[Proposition 2.10]{hausel_very_2022}}]\label{bblagrangians}
		In the above setting, the upward flows $W^+_\alpha$ for $\alpha \in X^{s\C^\times}$ are Lagrangian subvarieties of $X$.
	\end{proposition}
	
	\section{Moduli space of $G$-Higgs bundles}\label{sechiggs}
	
	\subsection{Main definitions, stability and moduli spaces}
	
	Let $C$ be a smooth projective complex curve of genus $g \ge 2$ with canonical line bundle $K_C$. Let $G$ be a connected semisimple complex group with Lie algebra $\lieg$.
	
	\begin{definition}
		A \textbf{$G$-Higgs bundle} over $C$ is a pair $(E,\varphi)$ where $E$ is a principal $G$-bundle over $C$ and $\varphi \in H^0(C,E(\lieg)\otimes K_C)$.
	\end{definition}
	
	Here $E(\lieg) := E \times_{\Ad} \lieg$ is the adjoint bundle, that is, the vector bundle associated to the adjoint representation of $G$. We are interested in studying the \textit{moduli space} which parametrises isomorphism classes of polystable $G$-Higgs bundles. For this we need to recall the suitable notions of stability \cite{garcia-prada_hitchin-kobayashi_2012}. We start with the following definition:
	\begin{definition}
		Let $\hat{G} \le G$ be a subgroup, and $E$ a principal $G$-bundle. A \textbf{reduction of structure group} of $E$ to $\hat{G}$ is a section $\sigma \in H^0(C,E(G/\hat{G}))$, where $E(G/\hat{G}) = E \times_{G} G/\hat{G}$.
	\end{definition}
	
	Given such a section $\sigma : X \to E(G/\hat{G})$, it is possible to pull back the $\hat{G}$-bundle $E$ on $E(G/\hat{G})$ which results in $E_{\sigma} := \sigma^*E$, a $\hat{G}$-bundle on $X$. This motivates the term \textit{reduction}. There is a canonical isomorphism $E_{\sigma} \times_{\hat{G}} G \simeq E$ and the map $E_{\sigma} = \sigma^*E \to E$ induced by the pullback gives a subvariety $E_{\sigma} \subseteq E$. 
	
	Fix a maximal compact subgroup $K \le G$. Let $\mathfrak{k}$ be its Lie algebra, which is a real subalgebra of $\mathfrak g$. Define for $s \in i\mathfrak k$ the spaces
	$$\lieg_s^0 = \{X \in \lieg : \Ad(e^{ts})(X) = X,\, \forall t \in \R\}, \quad \lieg_s = \{X \in \lieg : \Ad(e^{ts})(X) \text{ is bounded as } t \to \infty\},$$ 
	\noindent and the subgroups
	$$L_s = \{g \in G : \Ad(g)(s)=s\}, \quad P_s = \{g \in G : e^{ts}ge^{-ts} \text{ is bounded as } t \to \infty\}.$$ 
	These subgroups of $G$ have $\mathfrak g_s^0$ and $\mathfrak g_s$, respectively, as Lie algebras. $P_s$ is parabolic and $L_s$ is a Levi factor for $P_s$. We also define the character $\chi_s : \mathfrak g_s \to \mathbb C$ given by $\chi_s(x) = B(s,x)$, where $B$ is the Killing form on $\mathfrak g$.
	
	Now, for a $G$-bundle $E$ and a reduction $\sigma \in H^0(C, E(G/P_s))$ of structure group to $P_s$, we define the \textbf{degree} of the reduction. If a multiple $q \chi_s$ for some $q \in \Z_{>0}$ lifts to a character $\tilde{\chi}_s :P_s \to \mathbb C^*$, we set
	$$\deg E(\sigma, s) := \frac{1}{q}\deg (E_\sigma \times_{\tilde{\chi}_s} \mathbb C^*).$$
	It is also possible to define the degree using differential geometric techniques, as follows: there is a further reduction $\sigma'$ to $K_s := K \cap L_s$, the maximal compact of $L_s$. Let $A$ be a connection on $E_{\sigma'}$ and consider its curvature $F_A \in \Omega^2(X, E_{\sigma'}(\mathfrak k_s))$. We have that $\chi_s(F_A) \in \Omega^2(X, i \R)$, and the degree is defined as
	$$\deg E(\sigma,s) := \frac{i}{2\pi}\int_X\chi_s(F_A).$$
	
	We can now define stability (see \cite{garcia-prada_hitchin-kobayashi_2012}).
	
	\begin{definition}
		A $G$-Higgs bundle $(E,\varphi)$ is:
		\begin{itemize}
			\item \textbf{semistable}, if for any element $s \in i \mathfrak k$ and reduction $\sigma \in H^0(C, E(G/P_s))$ such that $\varphi \in H^0(E_\sigma(\lieg_s) \otimes K_C)$, we have $\deg E(\sigma, s) \ge 0$.
			\item \textbf{stable}, if it is semistable and, for any element $s \in i\mathfrak k$ and reduction of structure group $\sigma \in H^0(C, E(G/P_s))$ such that $\varphi \in H^0(E_\sigma(\lieg_s) \otimes K_C)$, we have $\deg E(\sigma, s) > 0$.
			\item \textbf{polystable}, if it is semistable and, for the $s \in i \mathfrak k$ and $\sigma \in H^0(C, E(G/P_s))$ such that $\varphi \in H^0(E_\sigma(\lieg_s) \otimes K_C)$ and we have $\deg E(\sigma, s) = 0$, there exists a reduction $\sigma' \in H^0(E_\sigma(P_s/L_s))$ of $E_\sigma$ to $L_s$ such that $\varphi \in H^0(E_{\sigma'}(\lieg_s^0) \otimes K_C)$.
		\end{itemize}
	\end{definition}
	
	There exists a \textbf{moduli space of isomorphism classes of polystable $G$-Higgs bundles over $C$} which is a complex algebraic variety of dimension $2(g-1)\dim(G)$. It was constructed in this generality by Simpson \cite{simpson_higgs_1992, simpson_moduli_1994}. We denote it by $\mathcal M(G)$.
	
	\begin{remark}\label{higgspairs}
		A more general construction using Geometric Invariant Theory can be found in the book of Schmitt \cite{schmitt2008geometric}. This construction includes stability and moduli spaces for \emph{$(G,V)$-Higgs pairs}, where $V$ is any representation of $G$ (not necessarily the adjoint one). In Section \ref{cstarmg} examples of such objects will appear to describe the fixed point locus of the $\C^\times$-action on $\mathcal M(G)$. However, we will not need to use these notions of stability explicitly.
		
		Similarly, the case where $G$ is not connected has been studied in \cite{garcia-prada_hitchin_kobayashi_nonconnected_unpublished}. Nevertheless, in the present work all the structure groups appearing in questions related to stability will be connected.
	\end{remark}
	
	An important feature of the moduli space $\mathcal M(G)$ is the existence of a completely integrable system on it \cite{hitchin_stable_1987}. It is defined by using that the ring of invariant polynomial functions $\C[\lieg]^G = \C[p_1,\dots,p_r]$ is a polynomial ring on $r = \rank \lieg$ generators.
	
	\begin{definition}\label{hitchinmap}
		The \textbf{Hitchin map} is defined by
	\begin{align*}
		h : \mathcal M(G) &\to \mathcal A(G) := \bigoplus_{i=1}^rH^0(C,K_C^{\deg p_i})\\
		(E,\varphi) &\mapsto (p_1(\varphi),\dots,p_r(\varphi)).
	\end{align*}
	\end{definition}
	
	The Hitchin map is a proper map of algebraic varieties, a fact which will appear later.

	In order to later study the Bia\l yinicki-Birula theory of the $\C^\times$-action on $\mathcal M(G)$, we now recall (see e.g. \cite[Theorem 2.3]{biswas_infinitesimal_1994}, \cite[Section 3.3]{garcia-prada_hitchin-kobayashi_2012}) the deformation theory for $G$-Higgs bundles, which allows to identify the tangent space of $\mathcal M(G)$ at its smooth points.
	
	\begin{definition}
		Let $(E,\varphi)$ be a $G$-Higgs bundle. Its \textbf{deformation complex} is the complex of sheaves defined by
		$$C^\bullet(E,\varphi) : E(\lieg) \xrightarrow{[\varphi, -]} E(\lieg) \otimes K_C.$$
	\end{definition}
	
	\begin{proposition}[{\cite[Theorem 2.3]{biswas_infinitesimal_1994}}]\label{deformationtheory}
		The space of infinitesimal deformations of a $G$-Higgs bundle $(E,\varphi)$ is naturally isomorphic to $\mathbb H^1(C^\bullet(E,\varphi)).$
	\end{proposition}
	
	The main idea for the proof, which we shall use later, is that under a suitable finite covering $C = \bigcup_{i \in I} U_i$, a hypercohomology class can be seen as a \v Cech cocycle $\{(s_{ij}, t_i)\}_{i,j \in I}$ such that $s_{ij} \in H^0(U_i \cap U_j, E(\lieg))$ are used to glue the trivial deformation $E|_{U_i} \times \Spec \C[\varepsilon]/\varepsilon^2$ by using the transition functions $1_G + \varepsilon s_{ij}$ on the intersections $U_i \cap U_j$, and $t_i \in H^0(U_i, E(\lieg) \otimes K_C)$ are used to deform the Higgs field by $\varphi'|_{U_i} := \varphi|_{U_i} + \varepsilon t_i$.
	
	The Bia\l ynicki-Birula description of the upward flows in terms of the weights at tangent spaces applies over smooth fixed points. In order to understand these, we need to recall the notion of \textit{simple} $G$-Higgs bundle.
	
	\begin{definition}
		A $G$-Higgs bundle $(E,\varphi)$ is \textbf{simple} if $\Aut(E,\varphi) = Z(G)$.
	\end{definition}
	
	Then, we have \cite[Proposition 3.18]{garcia-prada_hitchin-kobayashi_2012} the following smoothness result.
	
	\begin{proposition}\label{smoothness}
		If the $G$-Higgs bundle $(E,\varphi)$ is stable and simple, it defines a smooth point in $\mathcal M(G)$.
	\end{proposition}
	
	\subsection{$\C^\times$-action on $\mathcal M(G)$ and very stable Higgs bundles}\label{cstarmg}

	The moduli space $\mathcal M(G)$ is equipped with a natural action of $\C^\times$, where $\lambda \in \C^\times$ acts by scaling the Higgs field:
	$$(E,\varphi) \mapsto (E,\lambda \varphi).$$
	
	Notice that the Hitchin map $h_G$ from Definition \ref{hitchinmap} is $\C^\times$-equivariant with respect to the $\C^\times$-action on $\mathcal A(G)$ with weight $\deg p_i$ on the $H^0(C,K_C^{\deg p_i})$ factor. This, together with the properness of $h_G$, implies that $\mathcal M(G)$ is a semiprojective variety. The Bia\l ynicki-Birula theory for this $\C^\times$-action motivates the following definition.
	
	\begin{definition}[{\cite[Definitions 2.12 and 4.1]{hausel_very_2022}}]
		A smooth fixed point $(E,\varphi) \in \mathcal M(G)^{s\C^\times}$ is \textbf{very stable} if the only point in both its upward flow and the core is itself:
		$$W^+_{(E,\varphi)} \cap \mathcal C = \{(E,\varphi)\}.$$
	\end{definition}
	
	By Proposition \ref{verystableclosed}, a smooth fixed point $(E,\varphi)$ is very stable if and only if its upward flow $W^+_{(E,\varphi)}$ is closed. Moreover, by the $\C^\times$-equivariance of the Hitchin map $h_G$ and the fact that $h_G^{-1}(0)$ consists of the nilpotent polystable $G$-Higgs bundles, a smooth fixed point is very stable if and only if the only nilpotent Higgs bundle on its upward flow is the fixed point itself.
	
	Furthermore, there is a natural weight one symplectic structure $\omega = d\theta$ on the smooth locus $\mathcal M(G)^s$. The form $\theta$ is given at $(E,\varphi) \in \mathcal M(G)^s$ by identifying $T_{(E,\varphi)}\mathcal M(G)$ with $\mathbb H^1(C^\bullet(E,\varphi))$ via Proposition \ref{deformationtheory}, projecting to $H^1(C, E(\lieg))$ and using Serre duality to pair with $\varphi \in H^0(C,E(\lieg) \otimes K_C)$. As a consequence, Proposition \ref{bblagrangians} applies and the upward flows define Lagrangian subvarieties.
	
	We now describe the fixed points of the $\C^\times$-action \cite[Section 4.2]{biquard_arakelov-milnor_2021} (see also \cite{simpson_constructing_1988}). These are related to $\Z$-gradings of Lie algebras.
	
	\begin{definition}
		A \textbf{$\Z$-grading} of the Lie algebra $\lieg$ is a direct sum decomposition as vector spaces
		$$\lieg = \bigoplus_{j \in \Z}\lieg_j,$$
		\noindent such that $[\lieg_i,\lieg_j] \subseteq \lieg_{i+j}$.
	\end{definition}
	
	Given such a grading, the piece $\lieg_0 \subseteq \lieg$ is a Lie subalgebra and hence it has a corresponding connected subgroup $G_0 \subseteq G$. Since $\lieg$ is semisimple, there exists a \textbf{grading element} $\zeta \in \lieg_0$ such that
	$$\lieg_j = \{X \in \lieg : [\zeta, X] = jX\}.$$
	We have the description $G_0 = C_G(\zeta)$ as the centraliser of $\zeta$ in $G$. In order to see this, it suffices to justify that the centraliser of a semisimple element of $\lieg$ under the adjoint action of $G$ is connected. Let $S \subseteq G$ be the torus with Lie algebra $\left<\zeta\right> \subseteq \lieg$, which exists because $\zeta$ is semisimple. Then, we have $C_G(\zeta) = C_G(S)$. Now, the centraliser of a torus is always connected if $G$ is \cite[Theorem 22.3]{humphreys2012linear}, and we conclude. Furthermore, the fact that $[\lieg_0,\lieg_j] \subseteq \lieg_j$ implies that $G_0$ acts on each $\lieg_j$ by the adjoint representation.
	
	\begin{definition}
		Let $\hat{G} \subseteq G$ be a subgroup and $V \subseteq \lieg$ be a subspace which is invariant under the adjoint representation restricted to $\hat{G}$. A $G$-Higgs bundle $(E,\varphi)$ \textbf{reduces to a $(\hat{G},V)$-Higgs pair} if there exists a reduction of structure group $\sigma \in H^0(C, E(G/\hat{G}))$ of $E$ to $\hat{G}$ such that $\varphi \in H^0(C, E_\sigma(V) \otimes K_C) \subseteq H^0(C, E(\lieg) \otimes K_C).$
	\end{definition}
	
	As explained in Remark \ref{higgspairs}, there exists a notion of polystability for $(G_0,\lieg_j)$-Higgs pairs, and we have the following result (see e.g. \cite[Proposition 4.15]{biquard_arakelov-milnor_2021}).
	
	\begin{proposition}
		A polystable $G$-Higgs bundle $(E,\varphi) \in \mathcal M(G)$ is fixed by the $\C^\times$-action if and only if there exists a $\Z$-grading $\lieg = \bigoplus_{j \in \Z}\lieg_j$ such that $(E,\varphi)$ reduces to a polystable $(G_0,\lieg_j)$-Higgs pair for $j \neq 0$. 
	\end{proposition}

	Notice that, as can also be deduced from the $\C^\times$-equivariance of $h_G$, the above implies that fixed points are nilpotent (as there are finitely many pieces in the grading). Of course, there are nilpotent Higgs bundles which are not $\C^\times$-fixed. For example, those that reduce to $(G_0,\bigoplus_{j > 0}\lieg_j)$-Higgs pairs but do not reduce further to any $(G_0,\lieg_j)$-Higgs pair.
	
	\begin{example}\label{slnfixed}
		Consider the case of $G=\SL_n(\C)$. Via the standard representation, a $\SL_n(\C)$-Higgs bundle can be seen as a pair $(E,\varphi)$ where $E$ is a vector bundle of rank $n$ and trivial determinant over $C$ and $\varphi : E \to E \otimes K_C$ is a traceless twisted bundle endomorphism. In this case, it is known (as shown by Hitchin \cite[Section 7]{hitchin_self-duality_1987} and Simpson \cite[Section 4]{simpson_higgs_1992}, see also \cite[Lemma 9.2]{hausel_mirror_2003} and \cite[Section 3.1]{hausel_very_2022}) that fixed points are those that admit a direct sum splitting
		$$E = E_0 \oplus \dots \oplus E_{k-1},$$
		\noindent of vector subbundles $E_j \subseteq E$ with  $\rank E_j = r_j$, $\sum_{j=0}^{k-1}r_j = n$ and such that $\varphi(E_j) \subseteq E_{j+1} \otimes K_C$. These objects and their moduli spaces were studied in depth in \cite{garcia-prada_y_2013, garcia-prada_motives_2014}. In our setting, this corresponds to the $\Z$-grading of $\liesl_n(\C)$ given by choosing a splitting $\C^n = V_0 \oplus \dots \oplus V_{k-1}$ with $\dim V_j = r_j$ and letting
		$$\lieg_j := \{A \in \liesl_n(\C) : A(V_i) \subseteq V_{j+i}\}.$$
	\end{example}
	
	In this work we are interested in the fixed points with \textit{generically regular} Higgs field, that is, such that the centraliser of the Higgs field has (generically) the minimum possible dimension. In the case of $G=\SL_n(\C)$ from the example above, this corresponds to the situation where every $r_j = 1$, known in the literature as \textit{fixed points of type $(1,1,\dots,1)$}. These are the fixed points that were considered in \cite{hausel_very_2022} in the context of very stable Higgs bundles. We now define the analogue for $G$-Higgs bundles.
	
	Fix a maximal torus and Borel subgroup $T \subseteq B \subseteq G$, with Lie algebras $\liet \subseteq \mathfrak b \subseteq \lieg$. This determines a choice of simple roots $\Pi = \{\alpha_1,\dots,\alpha_r\} \subseteq \Delta := \Delta(\lieg, \liet)$ where $r = \rank \lieg = \dim \liet$. Given a root $\alpha \in \Delta$, we denote by $\lieg_{\alpha}$ the corresponding one-dimensional root subspace. We also denote by $\alpha^\vee$ the coroot corresponding to $\alpha$. We have the fundamental weights $\{\omega_1,\dots,\omega_r\} \subseteq \liet^*$ which are the dual basis to the simple coroots, and the fundamental coweights $\{\omega_1^\vee,\dots,\omega_r^\vee\} \subseteq \liet$ the dual basis to the simple roots.
	
	\begin{definition}
		The \textbf{Borel grading} of $\lieg$ is given by $\lieg_0 := \liet$ and 
		$$\lieg_j := \bigoplus_{\alpha \in \Delta : \heightr(\alpha) = j}\lieg_{\alpha}$$ \noindent for $j \neq 0$, where $\heightr(\alpha)$ denotes the height of $\alpha$ with respect to $\Pi$. Equivalently, it is given by the grading element $\zeta = \sum_{j=1}^r\omega_j^\vee$.
	\end{definition}
	
	The terminology refers to the fact that every $\Z$-grading corresponds to a parabolic subalgebra $\liep := \bigoplus_{j \ge 0}\lieg_j$, and in the Borel grading the corresponding subalgebra is precisely the Borel subalgebra $\mathfrak b \subseteq \lieg$.
	
	\begin{remark}
		In the Borel grading, we have that $G_0 = T$ and that the degree one part is the sum of the simple root spaces, i.e. $\lieg_1 = \bigoplus_i\lieg_{\alpha_i}$. The $G_0$-action on $\lieg_1$ has an open orbit $\Omega \subseteq \lieg_1$ given by elements $e = \sum_in_iX_{\alpha_i}$ with $n_i \neq 0$ for all $i$. which are the regular nilpotents of $\lieg$ belonging to $\lieg_1$. The elements of degree greater than one are not regular. Similarly, any other choice of grading does not have homogeneous regular nilpotent elements.
	\end{remark}
	Thus we will focus on the fixed point components given by the Borel grading and such that the Higgs field has degree one.
	\begin{definition}
		A fixed point $(E,\varphi) \in \mathcal M(G)^{\C^\times}$ is of \textbf{Borel type} if it reduces to a $(G_0,\lieg_1)$-Higgs pair for the Borel grading of $\lieg$ and $\varphi$ belongs to the open $G_0$-orbit $\Omega \subseteq \lieg_1$ generically.
	\end{definition}
	We can write fixed points of Borel type in terms of vector bundles for the classical groups.
	\begin{example}\label{typea}
		If $G = \GL_n(\C)$, $\PGL_n(\C)$ or $\SL_n(\C)$, fixed points of Borel type can be regarded as vector bundles $E = L_1 \oplus \dots \oplus L_{n}$ where $\rank L_i = 1$ with Higgs field given by $\varphi = (\delta_{i})_{i}$ where $\delta_{i} \in H^0(L_i^*L_{i+1}K_C)$, i.e. a choice of section for each simple root (notice here that we choose $T$ to be the diagonal matrices with coordinates $e_i$, and $e_{j+1}-e_j$ to be the positive roots, which is the opposite of some conventions in the literature). These sections must not be identically zero in order for the Higgs field to be generically in the open orbit. 
	\end{example}
	
	\begin{example}\label{typed}
		If $G = \SO_{2n}(\C)$, where for convenience we think of the bilinear quadratic form as $Q = (\delta_{i,2n-j+1})_{i,j}$ (here $\delta_{ab}$ is the Dirac delta), fixed points of Borel type are given by vector bundles
		$$E = (L_1 \oplus \dots \oplus L_{n}) \oplus (L_{n}^* \oplus \dots \oplus L_1^*),$$
		\noindent with $\rank L_j = 1$ and the Higgs field is determined by nonzero sections $\delta_i \in H^0(L_i^*L_{i+1}K_C)$ for $i \in \{1,\dots,n-1\}$ (these correspond to the simple roots $-e_i + e_{i+1}$) and $\eta \in H^0(L_{n-1}^*L_n^*K_C)$ (corresponding to the simple root $-e_{n-1} - e_n$). The Higgs field can be then written as
		$$\varphi = \begin{pNiceArray}{cccc|cccc}
			0        &        &          &       &           &        &           &   \\
			\delta_1 & \ddots &          &       &           &        &           &   \\
			& \ddots & 0        &       &           &        &           &   \\
			&        & \delta_{n-1} & 0     &           &        &           &   \\ \hline
			&        & \eta     & 0     & 0         &        &           &   \\
			&        &          & -\eta & -\delta_{n-1} & \ddots &           &   \\
			&        &          &       &           & \ddots & 0         &   \\
			&        &          &       &           &        & -\delta_1 & 0
		\end{pNiceArray}.$$
	\end{example}
	
	\begin{example}\label{typeb}
		Similarly, for $G = \SO_{2n+1}(\C)$ with bilinear quadratic form $Q = (\delta_{i,2n-j+2})_{i,j}$, fixed points of Borel type are
		$$E = (L_1 \oplus \dots \oplus L_{n}) \oplus \mathcal O_C \oplus (L_{n}^* \oplus \dots \oplus L_1^*),$$
		\noindent $\rank L_j = 1$ and the Higgs field is determined by nonzero sections $\delta_i \in H^0(L_i^*L_{i+1}K_C)$ for $i \in \{1,\dots,n-1\}$ for the simple roots $-e_i + e_{i+1}$ and $\eta \in H^0(L_{n}^*K_C)$ (corresponding to $-e_n$). The Higgs field can be written as
		$$\varphi = \begin{pNiceArray}{ccccccccc}
			0        &     &   &          &       &           &        &           &   \\
			\delta_1 & \ddots &          &       &      &     &        &           &   \\
			& \ddots & 0        &       &     &      &        &           &   \\
			&        & \delta_{n-1} & 0&     &           &        &           &   \\
			&        &     & \eta  &  0 &          &        &           &   \\
			& & & &-\eta & 0& &\\
			&        &    &      &  & -\delta_{n-1} & \ddots &           &   \\
			&        &   &       &       &           & \ddots & 0         &   \\
			&        &   &       &       &           &        & -\delta_1 & 0
		\end{pNiceArray}.$$
	\end{example}
	
	\begin{example}\label{typec}
		Finally we consider $G = \Sp_{2n}(\C)$ with symplectic form $Q = (\varepsilon_{i,j}\delta_{i,2n-j+1})_{i,j}$ where $$\varepsilon_{i,j} = \begin{cases}
			1 & i\le j\\
			-1 & i > j
		\end{cases}.$$ Fixed points of Borel type are
		$$E = (L_1 \oplus \dots \oplus L_{n}) \oplus (L_{n}^* \oplus \dots \oplus L_1^*),$$
		\noindent and the Higgs field is determined by nonzero sections $\delta_i \in H^0(L_i^*L_{i+1}K_C)$ for $i \in \{1,\dots,n-1\}$ corresponding to the simple roots $-e_i + e_{i+1}$ and $\eta \in H^0((L_{n}^2)^*K_C)$ corresponding to $-2e_n$. The Higgs field can be written as
		$$\varphi = \begin{pNiceArray}{cccc|cccc}
			0        &        &          &       &           &        &           &   \\
			\delta_1 & \ddots &          &       &           &        &           &   \\
			& \ddots & 0        &       &           &        &           &   \\
			&        & \delta_{n-1} & 0     &           &        &           &   \\ \hline
			&        &     & \eta     &   0       &        &           &   \\
			&        &          & & -\delta_{n-1} & \ddots &           &   \\
			&        &          &       &           & \ddots & 0         &   \\
			&        &          &       &           &        & -\delta_1 & 0
		\end{pNiceArray}.$$
	\end{example}
	
	The components $\delta_i,\eta$ of the Higgs field that appeared in the examples can be recovered generally as follows.
	
	\begin{definition}
		Let $(E,\varphi)$ be a fixed point of Borel type. For each simple root $\alpha_i \in \Pi$, we define the \textbf{$i$-th component} of $(E,\varphi)$ as:
		$$\delta_i := \pi_i(\varphi),$$
		\noindent where 
		$$\pi_i : H^0(E(\lieg) \otimes K_C) \to H^0(E(\lieg_{\alpha_i}) \otimes K_C)$$
		\noindent is the natural projection.
	\end{definition}
	
	Since $\lieg_1$ is the sum of simple root spaces, the components recover the Higgs field. Moreover, notice that the $E(\lieg_\alpha)$ are line bundles since $\dim \lieg_\alpha = 1$. Notice that the requirement that $\varphi$ generically belongs to the open $G_0$-orbit in $\lieg_1$ translates to all the components being nonzero.
	
	\begin{definition}\label{multvect}
		Let $(E,\varphi)$ be a fixed point of Borel type and $c \in C$. The \textbf{multiplicity vector of $(E,\varphi)$ at $c$} is the vector $(\ord_c \delta_{\alpha_i})_i \in \Z_{\ge 0}^r$ given by the orders of vanishing of each component.
	\end{definition}
	\subsection{Upward and downward flows on $\mathcal M(G)$}

	Now we describe the upward and downward flows for any smooth fixed point $(E,\varphi) \in \mathcal M(G)^{s\C^\times}$, not necessarily of Borel type. Recall from Proposition \ref{bbflows} that this amounts to studying the weights of the $\C^\times$-action induced on the tangent space at a fixed point. We will do so algebraically via the deformation theory of Proposition \ref{deformationtheory}, however there is an analytic proof \cite[Proposition 4.2]{collier_conformal_2018} in the $G=\GL_n(\C)$ case which can be readily adapted to the case of arbitrary semisimple structure group $G$. An algebraic proof also exists in the $G=\GL_n(\C)$ case \cite[Proposition 3.4]{hausel_very_2022}. 
	
	We start by considering the $\Z$-grading $\lieg := \bigoplus_j \lieg_j$ such that $(E,\varphi)$ reduces to a $(G_0,\lieg_k)$-Higgs pair. Let $\zeta \in \lieg_0$ be its grading element. Without loss of generality we may assume that $k=1$, otherwise we can use the subalgebra $\hat{\lieg} = \bigoplus_{j \in \Z} \lieg_{jk}$ and rescale the grading. Moreover, from the structure results about $\Z$-gradings of Lie algebras (e.g. \cite[Section 2.3]{biquard_arakelov-milnor_2021}) we may assume that $\zeta = \sum_{j=1}^rn_i\omega_i^\vee$ for $n_i \in \Z_{\ge 0}$. Then, as $\zeta$ is a coweight in $\liet$, it defines a cocharacter $\xi_{\bullet} : \C^\times \to G/Z(G) \simeq \Ad(G)$ of the adjoint group. Let $\mathfrak p := \bigoplus_{j \le 0}\lieg_j$ which is a parabolic subalgebra of $\lieg$ and $P \subseteq G$ the corresponding parabolic subgroup. We have the Levi decomposition $P \simeq U \rtimes G_0$, where $U$ is the unipotent radical of $P$. Consider the map:
	$$\pi_\lambda : P \to P$$
	\noindent given by $\pi_\lambda(p) := \xi_\lambda^{-1}p\xi_\lambda$. It is well defined even if $\xi_\lambda \in G^{\ad}$: we can choose any lift $\exp(t\zeta) \in G_0 \subseteq G$ where $t \in \C$ is a choice of logarithm with $\exp(t) = \lambda$, and then $\exp(-t\zeta)p\exp(t\zeta)$ is independent of the chosen $t$. We also use that $\liep$ is preserved by $\lieg_0$ so that the previous rule maps $P$ to itself. If $u \in U$ is in the unipotent radical, $\pi_\lambda(u) \to 1_G$ if $\lambda \to 0$. Indeed, on generators $\exp(X_\alpha)$ for negative roots $\alpha \in \Delta^-$ of degree $h := \alpha(\zeta) <0$ with respect to the $\Z$-grading, we have $$\pi_\lambda((X_\alpha)) = \exp(\Ad_{\xi_\lambda^{-1}}(X_\alpha)) = \exp(\lambda^{-h}X_\alpha)$$
	\noindent which limits to $1_G$ if $\lambda \to 0$. In other words, $\pi_\lambda$ interpolates between the identity (at $\lambda = 1$) and the projection to the Levi factor when limiting at $\lambda = 0$. Thus we define $\pi_0: P \to P$ to be the projection to the Levi $G_0 \subseteq P$. 
	
	Notice also that $P$ acts on $\lieg_{\le 1} := \mathfrak p \oplus \lieg_1$. Suppose that we have a $(P,\lieg_{\le 1})$-Higgs pair $(E',\varphi')$. We will show that $\pi_\lambda(E')$ acquires a natural Higgs field $\varphi'_\lambda$. For $\lambda \neq 0$, as $\pi_\lambda$ is given by conjugation by a group element, we have an isomorphism
	$$\psi_\lambda: \pi_\lambda(E') \rightarrow E',$$
	\noindent defined by regarding $\pi_\lambda(E')$ as the total space of $E'$ with action twisted by $\pi_\lambda$, choosing a logarithm $\exp(t) = \lambda$ and setting $\psi_\lambda(e) := e\exp(-t\zeta)$.
	The desired Higgs field is then $\varphi'_\lambda := \psi_\lambda^*(\lambda \varphi')$ and it does not depend on the chosen logarithm $t$. As can be seen in local coordinates, at $\lambda = 0$ the limiting Higgs field is obtained by projection to $\lieg_1$. In other words, the local expressions for $\varphi'_0$ are obtained by projecting the local expressions for any $\varphi'_\lambda$ to $\lieg_1$ (the result is independent of $\lambda$), which gives a well-defined section. Moreover, by construction we have for $\lambda \neq 0$ that $(\pi_\lambda(E'),\varphi'_\lambda) \simeq (E',\lambda \varphi')$ via $\psi_\lambda^*$.
	
	In particular, the above construction takes a $(P,\lieg_{\le 1})$-Higgs pair $(E',\varphi')$ and produces at $\lambda = 0$ a $(G_0,\lieg_1)$-Higgs pair, which we call the \textbf{associated graded} and denote by $\Gr(E',\varphi')$. That is:
	$$\Gr(E',\varphi') := (\pi_0(E'), \varphi'_0)$$
	This allows to describe the upward flow.
	
	\begin{proposition}\label{upwardflow}
		Let $(E,\varphi)$ be a smooth (stable and simple) $\C^\times$-fixed point corresponding to a $\Z$-grading of $\lieg$ with parabolic $\liep = \bigoplus_{j \le 0}\lieg_j$. Let $P \subseteq G$ be the corresponding parabolic subgroup. Its upward flow consists of the Higgs bundles $(E',\varphi') \in \mathcal M(G)$ that reduce to $(P,\lieg_{\le 1})$-Higgs pairs such that 
		$$\Gr(E',\varphi') \simeq (E,\varphi).$$
	\end{proposition}
	
	\begin{proof}
		One inclusion is given by the discussion preceding the proposition, which implies that $(E',\lambda\varphi') \simeq (\pi_\lambda(E'),\varphi'_\lambda)$ limits at $(\pi_0(E'), \varphi'_0) = \Gr(E',\varphi') \simeq (E,\varphi)$ when $\lambda \to 0$. 
		
		For the other, recall from the Bia\l ynicki-Birula theory in Proposition \ref{bbflows} that the upward flow is given by the positive weight subspace of the tangent space to $\mathcal M(G)$ at the fixed point. From Proposition \ref{deformationtheory}, the tangent space is obtained as the first hypercohomology of the complex
		$$C^\bullet: E(\lieg) \xrightarrow{[\varphi,-]} E(\lieg) \otimes K_C.$$
		
		Notice that $\varphi \in H^0(E(\lieg_1) \otimes K_C)$, the complex splits as the direct sum of the following complexes
		
		$$C_j^\bullet : E(\lieg_j) \xrightarrow{[\varphi,-]} E(\lieg_{j+1}) \otimes K_C,$$
		
		\noindent so that the tangent space splits naturally as
		
		$$T_{(E,\varphi)}\mathcal M(G) = \bigoplus_{j \in \Z} \mathbb H^1(C_j^\bullet).$$
		
		Now we claim that the induced $\C^\times$-action on $T_{(E,\varphi)}\mathcal M(G)$ has weight $-j$ on $\mathbb H^1(C_j^\bullet)$. Indeed, consider a deformation $(\tilde{E},\tilde{\varphi}) \in \mathbb H^1(C_j^\bullet)$. This means that locally $\tilde{\varphi} = \varphi + \varepsilon \varphi_{j+1}$ where $\varphi_{j+1}$ is a local section of $E(\lieg_{j+1}) \otimes K_C$ obtained from the \v Cech cocycle defining the hypercohomology class. Moreover, $\tilde{E}$ is obtained by glueing local copies of the trivial deformation of $E$ via the isomorphisms $1+\varepsilon s$ where the $s$ are local sections of $E(\lieg_j)$, also obtained from the cocycle. The $\C^\times$-action sends $(\tilde{E},\tilde{\varphi}) \mapsto (\tilde{E},\lambda \tilde{\varphi}) = (\tilde{E},\lambda \varphi + \varepsilon\lambda \varphi_{j+1})$. After applying the automorphism given by $\xi_\lambda$ (acting as $\Ad_{\xi_\lambda^{-1}}$ on the Lie algebra), the resulting bundle is $(\tilde{E}', \varphi + \varepsilon\lambda^{-j}\varphi_{j+1})$ where $\tilde{E}'$ is the deformation obtained from the sections $\lambda^{-j}s$ since they belong to $E(\lieg_j)$. In other words, at the tangent space $\lambda \in \C^\times$ acts as
		$$(\{s\}, \{\varphi_{j+1}\}) \mapsto (\{\lambda^{-j}s\}, \{\lambda^{-j}\varphi_{j+1}\}),$$
		\noindent as claimed.
		
		Thus, the upward flow is naturally identified with the deformations in $\bigoplus_{j < 0}\mathbb H^1(C_j^\bullet)$. These always result in a $P$-bundle (because the fixed point has structure group $G_0$ and we only consider infinitesimal automorphisms which belong to $E(\lieg_{\le-1})$) and the deformed Higgs fields are obtained by adding terms with values in $\lieg_{\le 0}$ to $\varphi$.
	\end{proof}
	
	\begin{example}
		In the $G=\SL_n(\C)$ case of Example \ref{slnfixed}, where the grading arises from a splitting $\C^n = V_0 \oplus \dots \oplus V_{k-1}$, we have that
		$$P = \{A \in \SL\nolimits_n(\C) : A(V_j) \subseteq \bigoplus_{i \le j} V_i\},$$
		\noindent in other words, $P$ consists of the automorphisms of $\C^n$ that preserve the filtration
		$$0 \subseteq V_0 \subseteq V_0 \oplus V_1 \subseteq \dots \subseteq \C^n,$$
		\noindent and $\liep \oplus \lieg_1$ are the endomorphisms of degree one with respect to the filtration. Therefore, the upward flow of a fixed point $(E = E_0 \oplus \dots \oplus E_{k-1}, \varphi)$ where $\varphi(E_j) \subseteq E_{j+1} \otimes K_C$ is given by the Higgs bundles $(E',\varphi')$ admitting a filtration by subbundles
		$$0 \subseteq E'_0 \subseteq E'_1 \subseteq \dots \subseteq E'_{k-1} = E',$$
		\noindent with $\varphi'(E_j') \subseteq E_{j+1}'\otimes K_C$ and such that the associated graded $(\bigoplus_j E'_j/E'_{j-1},\overline{\varphi'})$ is isomorphic to $(E,\varphi)$. This is \cite[Proposition 3.4]{hausel_very_2022}.
	\end{example}
	
	In a completely analogous manner (but changing the signs appropriately) the next proposition follows.
	
	\begin{proposition}
		Let $(E,\varphi)$ be a smooth (stable and simple) $\C^\times$-fixed point corresponding to a $\Z$-grading of $\lieg$ with parabolic $\liep = \bigoplus_{j \ge 0}\lieg_j$. Let $P \subseteq G$ be the corresponding parabolic subgroup. Its downward flow consists of the Higgs bundles $(E',\varphi') \in \mathcal M(G)$ that reduce to $(P,\lieg_{\ge 1})$-Higgs pairs such that 
		$$\Gr(E',\varphi') := (\pi_0(E'), \varphi'_0) \simeq (E,\varphi).$$
	\end{proposition}
	
	\begin{example}\label{uniformising}
		A key example of very stable Higgs bundle is the \textbf{canonical uniformising Higgs bundle}. We first recall \cite[Example 4.4]{biquard_arakelov-milnor_2021} its construction. Consider the regular nilpotent element $e:= \sum_{i=1}^rX_{\alpha_i}$. By the Jacobson--Morozov theorem, it can be completed to an $\mathfrak {sl}_2$-triple. More precisely, we can set $h:=2\sum_{i=1}^r \omega_i^\vee$ and there exists $f \in \lieg$ such that $(h,e,f)$ spans a copy of $\liesl_2(\C)$. Let $S \le G$ be the connected subgroup with Lie algebra $\gen{e,f,g} \subseteq \mathfrak g$, which can be isomorphic to $\PSL_2(\mathbb C)$ or $\SL_2(\mathbb C)$ depending on $G$ and the element $e$. Let $T \le S$ be the connected subgroup with Lie algebra $\gen{h}$. There are two cases:
		
		\begin{itemize}
			\item If $S \simeq \SL_2(\mathbb C)$, it is simply connected and thus the representation $\gen{h} \to \mathfrak{gl}(\gen{e})$ given by $\lambda h \mapsto \ad(\lambda h)$ lifts to a representation $\mathbb C^\times \simeq T \to \GL(\gen{e})$. As $[h,e] = 2e$, the $\mathbb C^\times$-action we get on $\gen{e}$ via this lift is $\lambda \cdot e = \lambda^{2}e$. Choose a square root $K_C^{\frac{-1}{2}}$ (this can be done as $\deg K_C = 2g-2$ is even) and let $E_T$ be the frame bundle for $K_C^{\frac{-1}{2}}$, in other words, the $\mathbb C^\times$-bundle such that the bundle associated to the standard representation of $\mathbb C^\times$ in $\mathbb C$ is $E_T(\mathbb C) \simeq K_C^{\frac{-1}{2}}$. Using the isomorphism $\mathbb C^\times \simeq T$ we have a $T$-bundle, and since $T$ acts on $\gen{e}$ with weight $2$ we have $E_T(\gen{e}) \simeq (K_C^{\frac{-1}{2}})^{2} = K_C^{-1}$. This means that $E_T(\gen{e}) \otimes K_C \simeq \mathcal O_C$ so we can define a constant section $e \in H^0(C, E_T(\gen{e}) \otimes K_C)$. Extending the structure group gives $(E_T(G),e)$ a $G$-Higgs bundle.
			
			\item If $S \simeq \PSL_2(\mathbb C)$, we can take its universal cover $\SL_2(\mathbb C) \to S$ which is of degree two. The torus $T \subseteq S$ lifts to $\hat{T} \subseteq \SL_2(\mathbb C)$. We have that $\hat{T}$ is a double cover of $T$ and there are isomorphisms with $\mathbb C^\times$ such that map $\hat{T} \simeq \mathbb C^\times \to T \simeq \mathbb C^\times$ is given by $\lambda \mapsto \lambda^2$. By the previous argument the adjoint action of $\hat{T} \simeq \mathbb C^*$ on $\gen{e}$ is given by $\lambda \cdot e = \lambda^{2} e$, so that it descends to $T$ as $\lambda \cdot e = \lambda e$. Now we let $E_T$ be the frame bundle of $K_C^{-1}$. The associated bundle is $E_T(\gen{e}) \simeq K_C^{-1}$ and hence $E_T(\gen{e}) \otimes K_C \simeq \mathcal O$. Thus $e$ defines a section of $E_T(\gen{e}) \otimes K_C$. Extending the structure group gives $(E_T(G),e)$ a $G$-Higgs bundle.
		\end{itemize} 
		
		By construction, in any case, the resulting $G$-Higgs bundle which we denote $(E,\varphi_0)$ reduces to a $(G_0,\lieg_1)$-Higgs pair of Borel type for the grading element $\zeta = \frac{h}{2} = \sum_{i=1}^r\omega_i^\vee$. Also note from the construction that its multiplicity vector (see Definition \ref{multvect}) equals zero. 
		
		We now claim that $(E,\varphi)$ is very stable. We first construct elements on its upward flow. This construction is due to Hitchin \cite[Section 5]{hitchin_lie-groups_1992} and produces a \textbf{Hitchin section}, that is, a section of $h_G$. We remark that, when $G$ is a classical group, this is different from the section constructed via companion matrices in \cite[Section 3]{hitchin_lie-groups_1992}: see \cite[Remark in Section 3]{hitchin_higgs_2013} and \cite{hameister_companion_2024}.
		
		Let $a = (a_1,\dots,a_r) \in \mathcal A(G)$ be a point in the Hitchin base and decompose $\lieg$ into irreducible representations of the $\liesl_2(\C)$ representation induced by $(h,e,f)$, getting
		$$\lieg = \bigoplus_{i=1}^r V_i.$$
		There are $r$ summands, and they can be ordered so that $\dim V_i = 2\deg p_i - 1$ where the $p_i \in \C[\lieg]^G$ are the chosen basis elements of invariant polynomials for the construction of $h_G$. Thus, if $f_i \in V_i$ is a lowest weight vector, we have that $[h,f_i] = (-2\deg p_i + 2) f_i$, and hence $E_T(\gen{f_i}) \otimes K_C \simeq K_C^{\deg p_i - 1} \otimes K_C \simeq K_C^{\deg p_i}$. This means that we have a well-defined constant section $f_i \in H^0(C,E_T(\gen{f_i}) \otimes K_C \otimes K_C^{-\deg p_i})$ and hence $a_if_i \in H^0(C, E(\lieg) \otimes K_C)$. This results in a well defined Higgs field
		$$\varphi_a := \varphi_0 + \sum_{i=1}^ra_if_i.$$
		
	 	The $f_i$ can be selected so that $h_G(E,\varphi_a) = a$. We also have that $(E,\varphi_a)$ are smooth points \cite[Section 5]{hitchin_lie-groups_1992} in $\mathcal M(G)$. By the $\C^\times$-equivariance of $h_G$, or directly from the construction and Proposition \ref{upwardflow}, we get the subvariety
	 	$$\mathcal A(G) \simeq \{(E,\varphi_a) : a \in \mathcal A(G)\} \subseteq W^+_{(E,\varphi_0)}.$$
		
		However, both sides are vector spaces of the same dimension, hence equality holds. It then follows from this explicit description of the upward flow that $(E,\varphi_0)$ is very stable.
	\end{example}
	
	\section{The affine Grassmannian and Hecke transformations}\label{sechecke}
	
	\subsection{Definitions and properties}
	One of the main techniques in the classification of very stable Higgs bundles of Borel type is, as in the $G=\GL_n(\C)$ case \cite{hausel_very_2022}, that of \textit{Hecke transformations}. These modify a Higgs bundle locally around a point to obtain a family of different Higgs bundles in the moduli space which are isomorphic to the starting one away from that point.
	
	The correct setting to study Hecke transformations, which allows to work with different groups $G$ as well as to perform more types of transformations (even in the $\GL_n(\C)$ case), is that of the \textit{affine Grassmannian}. In this section we start by recalling its definition and main properties that we will need for later use. Our main reference is \cite{zhu_introduction_2016}.
	
	We will define the affine Grassmannian via its functor of points, which is the natural point of view for Hecke transformations. The functor corresponds to the moduli problem of bundles on a disk isomorphic away from its centre. Let $D := \Spec \C[[z]]$ be the disk and $D^* := \Spec \C((z))$ be the punctured disk. Given a $\C$-algebra $R$, we can also consider the base changes $D_R := \Spec R[[z]]$ and $D_R^* := \Spec R((z))$ seen as families over $\Spec R$.
	
	\begin{definition}
		The \textbf{affine Grassmannian} $\Gr_G$ is defined by its $R$-points being pairs $(E,\beta)$ where $E$ is a principal $G$-bundle over $D_R$ and $\beta : E|_{D_R^*} \xrightarrow{\sim} D_R^* \times G$ is a trivialisation over $D_R^*$.
	\end{definition}
	
	The above functor of points defines an ind-projective ind-scheme \cite[Theorem 1.2.2]{zhu_introduction_2016}. Since we are only interested in the field of complex numbers, we may also regard it as an infinite-dimensional complex-analytic space.
	
	Another well-known definition is in terms of the \textit{loop groups}:

	\begin{definition}
		The \textbf{loop group} $LG$ and \textbf{positive loop group} $L^+G$ are given by their $R$-points being $G(\Spec R((z)))$ and $G(\Spec R[[z]])$, respectively.
	\end{definition}	
	
	We will sometimes denote $G((z)) := LG(\C)$ and $G[[z]] := L^+G(\C)$. We have \cite[Proposition 1.3.6]{zhu_introduction_2016} that $\Gr_G$ is isomorphic to the quotient $LG/L^+G$. Indeed, given an $R$-point $(E,\beta)$ and fixing a trivialisation over the entire disk $\varepsilon: D_R \times G \xrightarrow{\sim} E$ (this always exists locally on $R$), the map $\beta \circ \varepsilon$ is an automorphism of the trivial $G$-bundle on $D^*_R$, in other words an element $\sigma$ of $LG(R)$. Conversely such a loop $\sigma$ defines the point $(D_R \times G, \sigma)$ in the affine Grassmannian. Fixing from the start a different trivialisation $\varepsilon'$ changes everything by $(\varepsilon')^{-1} \circ \varepsilon$ which is a positive loop in $L^+G(R)$.
	
	Now fix a maximal torus and Borel subgroup $T \subseteq B \subseteq G$ with Weyl group $W_G = N_G(T)/T$. Given a cocharacter $\mu \in X_*(T) = \Hom(\mathbb G_m, T)$, it defines a map $\C((z))^{\times} \to LT(\C) \subseteq LG(\C)$. The image of $z$ defines a loop which we denote $z^\mu \in LG(\C)$, and hence an element of $\Gr_G$ which we denote similarly. Denoting by $X_*^+(T) \simeq X_*(T)/W_G$ the dominant coweights given by the choice of $B$, we have the \textbf{Cartan decomposition} due to Iwahori and Matsumoto \cite{iwahori_bruhat_1965}:
	$$G((z)) = \bigsqcup_{\lambda \in X^+_*(T)}G[[z]]z^{\lambda}G[[z]],$$
	\noindent whence we can identify
	$$\Gr_G(\C) = G((z))/G[[z]] = \bigsqcup_{\lambda \in X^+_*(T)}G[[z]]z^{\lambda}.$$
	
	Recall that there is a partial ordering in $X_*^+(T)$ given by $\lambda \le \mu$ if and only if $\mu - \lambda$ is a sum of simple coroots. Minimal elements under this ordering are called \textbf{minuscule}.
	
	\begin{definition}
		Given $\mu \in X_*^+(T)$, the $\Gr_\mu := G[[z]]z^\lambda$ in the decomposition above are called \textbf{(affine) Schubert cells}. The unions
		$$\Gr_{\le \mu} := \bigcup_{\lambda \le \mu}\Gr_{\lambda}$$ are called \textbf{(affine) Schubert varieties}.
	\end{definition}
	
	\begin{remark}
		 We may refer to Schubert cells and varieties for any coweight $\lambda \in X_*(T)$, meaning by this the corresponding notion for the dominant element in its Weyl group orbit. The Schubert cells may still be given by $G[[z]]z^{\lambda}$: if $w \in W_G$ is such that $w(\lambda)$ is dominant, and it is represented by $g \in N_G(T)$, then $z^{\lambda} = g^{-1}z^{w(\lambda)}g \in Gz^{w(\lambda)} \subseteq G[[z]]z^{w(\lambda)}$.
	\end{remark}
	
	Affine Schubert cells are smooth quasi-projective varieties whose closure is the corresponding affine Schubert variety, hence the latter are projective. The fact that Schubert cells are not closed (unless they correspond to a minuscule cocharacter) will play a key role in the classification of very stable $G$-Higgs bundles. Thus we will need the following statement \cite[Proposition 2.1.5]{zhu_introduction_2016} and its proof, explicitly.
	
	\begin{lemma}\label{grassmanniancurve}
		Let $\mu \in X_*(T)$ be a cocharacter and $\alpha^\vee \in \Phi^\vee_+$ a positive coroot with corresponding positive root $\alpha \in \Phi^+$. Assume that $\alpha(\mu) - 1 \ge 0$. Then, there exists a curve $C_{\mu, \alpha} \subseteq \Gr_G$, isomorphic to $\mathbb P^1_\C$, such that $C_{\mu,\alpha} \cap \Gr_{\mu-\alpha} = \{z^{\mu-\alpha}\}$ and the remaining $C_{\mu,\alpha} \setminus \{z^{\mu-\alpha}\} \simeq \mathbb A^1_\C$ is contained in $\Gr_{\mu}$.
	\end{lemma}
	
	\begin{proof}
		Let $m := \alpha(\mu) - 1 \ge 0$ and $i_\alpha : \SL_2 \hookrightarrow G$ the corresponding inclusion of the $\SL_2$-subgroup for $\alpha$. The latter naturally induces $Li_\alpha : L\SL_2 \to LG$. Consider the following subgroup isomorphic to $\SL_2$:
		\[K_{m} := \left\{\begin{pmatrix}
			a & bz^m \\
			cz^{-m} & d
		\end{pmatrix} : a,b,c,d \in \C,\,ad-bc = 1\right\} \subseteq L\SL\nolimits_2.\]
		The desired curve is $Li_\alpha(K_m)z^\mu$. The claims can be checked by considering the subgroup $K_m \cap L^+\SL_2$ given by $c=0$. By multiplication with a diagonal element in $\SL_2$ (which commutes with $\mu$ and vanishes in the quotient $\Gr_G$) we get the affine part $Li_\alpha\begin{pmatrix}
			1 & bz^m \\ 0 & 1
		\end{pmatrix}z^\mu \simeq \mathbb A^1$, contained in $Li_\alpha(\SL_2[[z]])z^\mu \subseteq G[[z]]z^\mu$. On the other hand, the remaining point can be computed \cite[Proposition 2.1.5]{zhu_introduction_2016} to be $Li_\alpha\begin{pmatrix}
			0 & -z^m \\
			z^{-m} & 0
		\end{pmatrix}z^\mu = z^{\mu-\alpha^\vee}$.
	\end{proof}

	Since we are interested in $G$-Higgs bundles, we now bring Higgs fields into the picture of the affine Grassmannian.  Consider $(E,\beta)$ a principal $G$-bundle on the disk $D$ with a trivialisation $\beta: E|_{D^*} \xrightarrow{\sim} D^* \times G$. If $E$ is now endowed with an untwisted Higgs field $\varphi \in H^0(D, E(\lieg))$, via the trivialisation we obtain an element $X \in H^0(D^*, (D^* \times G)(\lieg)) = \Hom(D^*,\lieg) \simeq \lieg \otimes_{\C} \C((z)) =: \lieg((z)).$
	
	\begin{definition}
		Let $X \in \lieg((z))$. The moduli space of pairs $((E,\varphi),\beta)$ consisting of a Higgs bundle $(E,\varphi)$ on $D$ and a trivialisation $\beta : (E,\varphi)|_{D^*} \xrightarrow{\sim} (D^* \times G, X)$ is called the \textbf{affine Springer fibre} of $X$ and denoted by $\Gr^X_G$.
	\end{definition}
	
	Note that we can map the affine Springer fibre as a subspace of the affine Grassmannian by forgetting the Higgs field. The image are the pairs $(E,\beta)$ of principal $G$-bundle and trivialisation such that $\beta^*X \in H^0(D^*, E(\lieg))$ extends to a Higgs bundle $\varphi$ over the whole disk $D$. It is then clear that under the loop group interpretation of $\Gr_G$ we have
	\[\Gr^X_G = \left\{[\sigma] \in \Gr_G : \Ad_{\sigma^{-1}}X \in \lieg[[z]]\right\} \subseteq \Gr_G.\]

	\subsection{Hecke transformations}\label{heckeprin}
	
	In this section we define the action of the affine Springer fibres on Higgs bundles via Hecke transformations. We start by considering the case of principal $G$-bundles, developed in \cite{wong_hecke_2013}. Fix a point $c \in C$ and let $C_0 := C \setminus \{c\}$ be its complement.
	
	\begin{definition}
		A \textbf{Hecke transformation} of a principal $G$-bundle $E$ at $c \in C$ is a pair $(E',\psi)$ where $E'$ is another principal $G$-bundle over $C$ and
		$$\psi : E'|_{C_0} \xrightarrow{\sim} E|_{C_0}$$
		\noindent is an isomorphism.
	\end{definition}
	
	\begin{definition}
		Two Hecke transformations $(E',\psi')$, $(E'',\psi'')$ are said to be \textbf{equivalent} if there is an isomorphism $\alpha : E' \to E''$ whose restriction to $C_0$ verifies $\psi'' \circ \alpha = \psi'$.
	\end{definition}
	
	Hecke transformations for principal $G$-bundles are strongly related to $LG$ and $\Gr_G$. In order to see this, let $(E',\varphi')$ be a Hecke transformation for $E$ and fix a disk neighbourhood $C_1 \subseteq C$ around $c$ and trivialisations $t_1: E|_{C_1} \to C_1 \times G$ for $E$ and $t_1'$ for $E'$. Fix also trivialisations $t_0, t_0'$ over $C_0$, which exist since $G$ is connected and semisimple (cf. \cite{harder_halbeinfache_1967}).

	Let $C_{01} := C_0 \cap C_1$. Via the previous trivialisations, we obtain transition functions
	$$f_E, f_E' : C_{01} \to G,$$
	\noindent defined such that $t_0 \circ t_1^{-1}(p,g) = (p, f_E(p)g)$ for $p \in C_{01}$, and analogously for $E'$. From the fact that $t_1 \circ \psi|_{C_{01}} \circ (t_1')^{-1}$ is an automorphism of the trivial $G$-bundle over $C_{01}$, we obtain a section
	$$\sigma : C_{01} \to G,$$
	\noindent which can be regarded as an element of $LG$ for a suitable choice of $C_1$ (a disk). Notice that the defining property for $\sigma$ is that
	$$f_{E'} = f_E\cdot\sigma$$
	\noindent over the punctured neighbourhood $C_{01}$.
	
	Conversely, given the trivialisation $\{(C_0,t_0), (C_1, t_1)\}$ for $E$ as above, a choice of $\sigma : C_{01} \to G$ produces a Hecke transformation $(E',\psi)$ where $E'$ is given by the transition function $f_{E'} := f_E \cdot \sigma$. If $(C_i,t_i')$ are trivialisations of $E'$ with said transition function then we recover $\psi$ as $t_0^{-1} \circ t_0'$.
	
	Lastly, if we fix $E$ and its trivialisation $\{(C_0,t_0), (C_1, t_1)\}$ as above, we want to investigate when two Hecke transformations $\sigma',\sigma'' : C_{01} \to G$ are equivalent. If this is the case, we can choose trivialisations $t_i', t_i''$ such that $f_{E'} = f_E\sigma$, $f_{E''} = f_E\sigma''$, $t_0'=t_0''$ and $t_1' = \beta t_1''$ for $\beta : C_{1} \to G$. This implies that
	$$\sigma'' = \sigma'\beta,$$
	\noindent so that both loops $\sigma',\sigma''$ differ by a positive loop $\beta$. We thus conclude the following as a result of the discussion throughout the section.
	
	\begin{proposition}[{\cite[Section 1.4]{wong_hecke_2013}}]
		Let $c \in C$. Let $E$ be a principal $G$-bundle with a fixed trivialisation $\{(C_i,t_i)\}$ over $C_0 := C \setminus c$ and a disk neighbourhood $C_1$ of $c$. The space of Hecke transformations at $c$ of $E$ is in correspondence with the affine Grassmannian $\Gr_G = LG/L^+G$.
	\end{proposition}
	
	\begin{remark}
		The identification is not canonical as it depends on the fixed trivialisation for $E$. When using Hecke transformations, we will always have such a trivialisation fixed beforehand so that we can work directly with the affine Grassmannian. 
		
		Moreover, a change of trivialisation results in a left multiplication by a positive loop $\gamma : C_1 \to G$, so that at any rate the $L^+G$-orbit of the Hecke transformation is well defined. In particular, we can associate an invariant $[(E',\psi)] \in X^+_*(T)$ to the Hecke transformations, defined so that the corresponding $\sigma$ lies in the Schubert cell $\Gr_{[(E',\psi)]}$. This is known as the \textbf{type} of the transformation.
	\end{remark} 
	
	\begin{example}
		Consider $G=\PGL_n(\C)$ and let $E := \mathcal O_C \oplus K_C^{-1}$. Fix a trivialisation as above which splits as a direct sum of trivialisations for each of the two line bundles. Let $\sigma = z^{\omega_1^\vee}$ be the first fundamental cocharacter, which can be seen as
		$$z \mapsto \begin{pmatrix}
			z & 0 \\
			0 & 1
		\end{pmatrix}$$
		\noindent after a choice of uniformiser $z$ of the disk $C_1$. Taking the constant map $z \mapsto 1$ as a transition function valued in $\C^\times$ produces the trivial line bundle $\mathcal O_C$, and taking the identity $z \mapsto z$ as a transition function from $C_1$ to $C_0$ corresponds to the line bundle $\mathcal O_C(-c)$. Therefore, the resulting Hecke transformation is isomorphic to $E' = \mathcal O_C(-c) \oplus K_C^{-1}$, where $\psi$ is the sum of the identity on the $K_C^{-1}$ factor and the isomorphism $\mathcal O_C(-c)|_{C \setminus c} \simeq \mathcal O_{C \setminus c}$ given by the induced  trivialisation on $\mathcal O_C(-c)$.
	\end{example}
	
	\begin{remark}\label{toptype}
		It is worth noting that, as it is clear from the previous example, Hecke transformations do not preserve the topological invariant $\varepsilon(E) \in \pi_1(G)$ of the bundle. In fact, we have \cite[Section 1.3.1]{wong_hecke_2013}
		$$\varepsilon(E') = \varepsilon(E) + [\sigma],$$
		\noindent where $[\sigma] \in X_*^+(T)$ is the Schubert type of the transformation seen as an element in $\pi_1(G)$.
	\end{remark}
	
	\begin{remark}\label{minusculeflagvar}
		When $G=\GL_n(\C), \SL_n(\C), \PGL_n(\C)$, there is a well known notion of Hecke transformation (which is used, for example, in the classification of very stable $\GL_n(\C)$-Higgs bundles of type $(1,1,\dots,1)$ \cite[Section 4.2]{hausel_very_2022}) involving choosing a vector subspace of the fibre $E|_c$ and performing some sheaf operations. This can be explained as follows in our setting. Whenever $\omega_i^\vee \in X_*^+(T)$ is a minuscule cocharacter (recall that this means minimal with respect to the partial ordering on the space of dominant coweights or, equivalently, such that the Weyl group acts transitively on the weights of the irreducible representation $\rho_i : G^\vee \to \GL(V_i)$; this automatically implies that it is fundamental), the map
		$$\Gr_{z^{\omega_i^\vee}} = L^+Gz^{\omega_i^\vee}= L^+G/(L^+G \cap z^{-\omega_i^\vee}(L^+G)z^{\omega_i^\vee}) \to G/P_i,$$
		\noindent given by evaluation at the centre of the disk, where $P_i \subseteq G$ is the maximal parabolic associated to $\omega^\vee_i$, is an isomorphism (see \cite[Example 4.12]{hurtubise_elliptic_2003} or \cite[Lemma 2.1.13]{zhu_introduction_2016}). 
		
		Therefore, Hecke transformations of minuscule type are given by elements in the generalised flag varieties $G/P_i$. In particular, in the case of Dynkin type $A$, every $\omega_i^\vee$ is minuscule, $G/P_i$ is isomorphic to a Grasmannian in $E|_c$ and hence we get Hecke transformations for each subspace of $E|_c$ of any dimension. This notion agrees with the standard one.
	\end{remark}
	
		Now we wish to perform Hecke transformations for $G$-Higgs bundles. They are defined similarly as follows.
	
	\begin{definition}
		A \textbf{Hecke transformation} of a $G$-Higgs bundle $(E,\varphi)$ at $c \in C$ is a tuple $(E', \varphi', \psi)$ where $(E',\varphi')$ is another $G$-Higgs bundle over $C$ and $\psi$ is a Higgs bundle isomorphism of their restrictions to $C_0$.
		
		Two Hecke transformations are equivalent if there is a Higgs bundle isomorphism between them which, when restricted to $C_0$, is compatible with the isomorphisms to $E|_{C_0}$.
	\end{definition}
	
	There is a forgetful functor from the Hecke transformations of $(E,\varphi)$ to those of $E$, so the space of Hecke transformations for the Higgs bundle should be a subspace of $\Gr_G$ which we identify now.
	
	We fix a trivialisation $\{(C_0,t_0)$, $(C_1,t_1)\}$ for $E$ as before and a local uniformiser $z$ in $C_1$. This gives an identification of the space of Hecke transformations of $E$ with $\Gr_G$. We want the Higgs field $\varphi \in H^0(E(\lieg) \otimes K_C)$ to induce a Higgs field $\varphi' \in H^0(E'(\lieg) \otimes K_C)$ on $E'$. 
	
	Over $C_0$, using the trivialisation, we can regard the Higgs field as a regular function $\varphi_0 : C_0 \to \lieg$. Indeed, under the identification $E|_{C_0} = C_0 \times G$, the image point of$\varphi$ at $p \in C_0$ can be written as an element of $E\times_{\Ad}\lieg$ as $[(p,g),\Ad_{g^{-1}}\varphi_0(p)]$. We define $\varphi_0' := \varphi_0 : C_0 \to \lieg$. The reason is that the Hecke transformation $\psi$ given by $\sigma$ equals the identity on the chosen trivialisations for $E$ and the induced trivialisation for $E'$, so this is the only choice preserving the Higgs field.
	
	Over $C_1$ we consider the regular function $\varphi_1 : C_1 \to \lieg$, defined as above over the trivialisation $C_1$. We define $\varphi_{01}' : C_{01} \to \lieg$ by the rule
	$$\varphi_{01}' := \Ad_{\sigma^{-1}}\varphi_1|_{C_{01}},$$
	\noindent where $\sigma \in \Gr_G$ is the loop giving the Hecke transformation $E'$. This gives a rational function $\varphi_{01}'$  from $C_1$ to $\lieg$ with $c$ as its only (possible) pole. Thus $\varphi_{01}' \in \lieg((z))$. We now introduce the extra assumption
	$$\Ad_{\sigma^{-1}}\varphi_1 \in L^+\lieg$$
	\noindent that is, the assumption that $\varphi_{01}'$ can be extended at the puncture to define a regular map $\varphi_1' : C_1 \to \lieg$.
	
	By construction it is clear that the data $\{(C_0,t_0',\varphi_0'), (C_1, t_1', \varphi_1')\}$ defines a Higgs field on $E'$. Here the $t_j'$ are the trivialisation functions for $E'$ which verify $t_0' \circ (t_1')^{-1} = (t_0 \circ t_1^{-1}) \cdot \sigma$. Indeed, using $G$-equivariance to regard $\varphi_0'$ and $\varphi_1'$ as functions over $C_{01} \times G$, and denoting by $R_g$ the right action by $g$ on the bundle, we compute over $C_{01}$ that
	\begin{align*}
		(t'_0)^*\varphi_0' &= (t'_0)^*\varphi_0 = (R_\sigma \circ t_0 \circ t_1^{-1} \circ t_1')^*\varphi_0  \\&= (t_1^{-1} \circ t_1')^*t_0^*\Ad_{\sigma^{-1}} \varphi_0 =(t_1^{-1} \circ t_1')^*t_1^*\Ad_{\sigma^{-1}} \varphi_1 = (t_1')^*\varphi'_1,
	\end{align*}

	\noindent where we used the $G$-equivariance of $\varphi_0$ and the fact that $t_0^*\varphi_0 = t_1^*\varphi_1$ over $C_{01}$.
	
	Furthermore, this calculation shows that if $(E',\varphi')$ is a Hecke transformation of $(E,\varphi)$ given by $\sigma$ in the fixed trivialisations, we must have that $\varphi_1'|_{C_{01}} = \Ad_{\sigma^{-1}}\varphi_1$ so that $\Ad_{\sigma^{-1}}\varphi_1$ has to be a positive loop. In other words:
	
	\begin{proposition}
		The space of Hecke transformations at $c$ of $(E,\varphi)$ is isomorphic to the affine Springer fibre
		$$\Gr_G^{(\varphi_1)_c} = \{\sigma \in \Gr_G : (\Ad_{\sigma^{-1}}\varphi_1)_c \in L^+\lieg\} \subseteq \Gr_G,$$
		\noindent where the subscript $c$ means the germ at $c$ seen as an element of $L\lieg$ via the fixed trivialisation.
	\end{proposition}
	
	After identifying the space of Hecke transformations at a point $c \in C$ with an affine Springer fibre (i.e. after fixing a suitable trivialisation), we will denote by $\mathcal H^c_\sigma(E,\varphi)$ (or $\mathcal H_\sigma(E,\varphi)$ if the point is clear from context) the Hecke transformation corresponding to $\sigma \in \Gr_G^{(\varphi_1)_c}$. Similarly, we will denote the space of Hecke transformations by $\mathcal H^c_{(E,\varphi)}$ (or $\mathcal H_{(E,\varphi)}$ if the point is clear from context) and we will implicitly identify it with $\Gr_G^{(\varphi_1)_c}$ in the presence of a fixed trivialisation as above.
	
	\begin{remark}\label{minusculefibre}
		Suppose that $\omega_i^\vee$ is minuscule. Then, by Remark \ref{minusculeflagvar}, its Schubert cell is isomorphic to a flag variety $G/P_i$. Thus, the space of Hecke transformations with type $z^{\omega_i^\vee}$ becomes
		$$\{gP_i \in G/P_i : (\Ad_{z^{-\omega_i^\vee}g^{-1}}\varphi_1)_c \in L^+\lieg\}.$$
		
		In this case $gP_i$ belongs to the Hecke transformation space if and only if
		$$\varphi_1|_c \in L^+\lieg \cap \Ad_{g\omega_i^\vee} L^+\lieg.$$
		
		The expression to the right is the Lie algebra of the parabolic $gP_i^{opp}g^{-1}$ at $z=0$ (the superscript \textit{opp} denotes the opposite parabolic). In other words, the space of Hecke transformations with type $\omega_i^\vee$ can be written as
		$$\{gP_i \in G/P_i : \varphi_1(c) \in \Ad_g \liep_i^{opp}\},$$
		
		\noindent which is the \textbf{(partial) Springer fibre} of $\varphi_1(c)$ corresponding to the parabolic $P_i^{opp}$ (here we use that two parabolics are conjugate if and only if their opposites are, so $G/P_i$ also parameterises the parabolics conjugate to $P_i^{opp}$).
	\end{remark}
	
	\begin{remark}\label{minusculekostantsection}
		Let $e \in \lieg$ be a regular nilpotent element and $(h,e,f)$ a (principal) $\liesl_2$-triple containing $e$. Consider the Kostant section $\mathcal S := e + C_{\lieg}(f)$ which is a section for the Chevalley map $\chi : \lieg \to \lieg \git G \simeq \liet / W$, where $W$ is the Weyl group for $(\lieg, \liet)$. Now let $\liep \subseteq \lieg$ be a parabolic subalgebra containing $\liet$, $\liel$ its levi factor and $W_\liel$ the Weyl group for $(\liel,\liet)$. Let $\tilde{\mathcal S}_P$ be the preimage of $\mathcal S$ under the Grothendieck--Springer resolution for $P$. Then $\tilde{\mathcal S}_P \to \mathcal S$ is isomorphic to $\liet / W_\liel \to \liet / W$. In particular, the Springer fibre at $e$ for $P$ contains a single point.
		
		Combining this with Remark \ref{minusculefibre}, we get that if $\varphi_1(c) \in \lieg$ is a regular nilpotent element then there is only one possible Hecke transformation of type a minuscule (in the Langlands dual) cocharacter $\omega_i^\vee$. This is a generalisation of the fact that the matrix
		$$\begin{pmatrix}
			0 & 0 & \dots & 0 & 0 \\
			1 & 0 & \dots & 0 & 0 \\
			0 & 1 & \dots & 0 & 0 \\
			\vdots & \vdots & \ddots & \vdots & \vdots \\
			0 & 0 & \dots & 1 & 0 \\
		\end{pmatrix}$$
		\noindent only admits one invariant subspace of a given dimension.
	\end{remark}
	
	We conclude this section with an example.
	
	\begin{example}
		Consider the canonical uniformising $\PGL_3(\C)$-Higgs bundle from Example \ref{uniformising}: $E = K_C \oplus \mathcal O \oplus K_C^{-1}$ and $\varphi = \begin{pmatrix}
			0 & 0 & 0 \\
			1 & 0 & 0 \\
			0 & 1 & 0
		\end{pmatrix}$. We take as Cartan subalgebra the diagonal matrices with coordinates $(e_1,e_2,e_3)$ (subject to $e_1+e_2+e_3=0$) and as simple roots $\alpha_1 := e_2 - e_1$ and $\alpha_2 = e_3 - e_2$. A choice of simple root vectors is $X_1 := E_{21}$ and $X_2 := E_{32}$, where $E_{ij}$ denotes the matrix with a one on the entry $(i,j)$ and zeroes elsewhere. This example is of Borel type as the Higgs field has constant value equal to $\alpha = X_1 + X_2 \in \lieg_1$ and the bundle reduces to $G_0$ the maximal torus of diagonal matrices. Consider the first fundamental cocharacter, which as a cocharacter of $T \subseteq G$ is:
		$$z^{\omega_1^\vee} = \begin{pmatrix}
			z^{-1} & 0 & 0 \\ 0&  1 & 0 \\ 0& 0& 1 
		\end{pmatrix},$$
		\noindent and as an element of $\liet \subseteq \liesl_3(\C)$ it is $\omega_1^\vee = \mathrm{diag}(-2/3, 1/3, 1/3)$. Now, let $z = \exp(t)$ and use the fact that $[\omega_i^\vee, X_j] = \alpha_j(\omega_i^\vee)X_j = \delta_{ij}X_j$ to compute:
		$$\Ad_{z^{k\omega_i^\vee}}(X_j) = \Ad_{\exp(tk\omega^\vee_i)}(X_j) = \exp(\ad(tk\omega^\vee_i))(X_j) = \begin{cases}
			z^kX_j & i=j,\\
			X_j & i \neq j
		\end{cases}.$$
		
		Thus, in this case, the choice $\sigma = z^{\omega_1^\vee}$ is not in the space of Hecke transformations since $\sigma^{-1} = z^{-\omega_1^\vee}$ maps the Higgs field to $\frac{1}{z}X_1 + X_2$ which has a pole. However we can take $\sigma = z^{-\omega_1^\vee}$. Then the Hecke modified Higgs bundle is
		$$E' = K_C(-c) \oplus \mathcal O \oplus K_C^{-1},$$
		\noindent since the transition function for $\mathcal O(-c)$ from the trivialisation $C_1$ to $C_0$ is precisely $z$.
		Moreover, the new Higgs field in local coordinates around $z$ is, according to the previous calculations, equal to $zX_1 + X_2$. If $s_c \in H^0(\mathcal O(c))$ denotes the canonical section, we can globally write it as
		$$\varphi' = \begin{pmatrix}
			0 & 0 & 0 \\
			s_c & 0 & 0 \\
			0 & 1 & 0
		\end{pmatrix}.$$
		
		Thus we see how this approach recovers the one in \cite[Section 4.2]{hausel_very_2022}. Note that we can always revert the Hecke transformation by taking $\sigma^{-1}$ in the space of Hecke transformations of $(E',\varphi')$ with respect to the induced trivialisation.
		
		Here is an example of Hecke transformation that produces a non-fixed nilpotent. We take $\sigma^{-1} = \gamma_+^{-1}\omega_1^\vee$ where $\gamma_+$ is the following element of the \textit{root subgroup}:
		$$\gamma_+ = \begin{pmatrix}
			1 & 0 & 0 \\
			0 & 1 & 0 \\
			0 & 1 & 1
		\end{pmatrix}.$$
		
		Then $(\omega_1^\vee)^{-1}$ sends $X_1 \mapsto zX_1$, $X_2 \mapsto X_2$ and $\gamma_+$ sends $X_1 \mapsto X_1 + X_3$, $X_2 \mapsto X_2$. So in the Hecke transformation, locally (on $C_1$) the Higgs field has value $zX_1+X_3 + X_2$ which is nilpotent but not graded (also $G_0$ is not preserved by this $\sigma$). 
	\end{example}
	
	\subsection{$\C^\times$-action on the affine Grassmannian}
	We introduce a $\C^\times$-action on the affine Grassmannian that will later allow us to study the Bia\l ynicki-Birula decomposition of $\mathcal M(G)$ at fixed points of Borel type by applying Hecke transformations at the fixed point. Recall the grading element $\zeta = \sum_{j=1}^r\omega_j^\vee$ corresponding to fixed points of Borel type and the corresponding cocharacter $\xi_{\bullet} : \C^\times \to \Ad(G)$ that it defines. Notice that left multiplication by such a cocharacter gives a well defined action on $\Gr_G$.
	
	\begin{proposition}\label{cstargrass}
		Let $(E,\varphi)$ be a fixed point of the $\C^\times$-action of Borel type with fixed trivialisation $\{(C_0,t_0)$, $(C_1,t_1)\}$ so that the Hecke transformation space is identified with the affine Springer fibre:
		$$\mathcal H_{(E,\varphi)} := \{\sigma \in \Gr_G : (\Ad_{\sigma^{-1}}\varphi_1)_c \in L^+\lieg\} \subseteq \Gr_G.$$
		
		We define $\sigma_\lambda := \xi_\lambda^{-1}\sigma \in \Gr_G$. Then
		\begin{enumerate}
			\item The map $\sigma \mapsto \sigma_\lambda$ defines a $\C^\times$-action on $\mathcal H_{(E,\varphi)}$ preserving the Schubert cells.
			\item There is an isomorphism $\psi_\lambda$ between the Hecke transformations $\mathcal H_{\sigma_\lambda}(E,\varphi)$ and $\lambda\mathcal H_{\sigma}(E,\varphi)$. 
			\item Choose a logarithm $t \in \C$ with $\exp(t) = \lambda$. There are induced trivialisations for $\mathcal H_{\sigma_\lambda}(E,\varphi)$ and $\lambda \mathcal H_{\sigma}(E,\varphi)$ where $\psi_\lambda$ is given by left multiplication with $\exp(-t\zeta)$ on $C_0$ and by the identity on $C_1$.
		\end{enumerate} 
	\end{proposition}

	We finish the section with the proof of the above proposition, which is a local verification. 
	
	\begin{proof}
	Fix $\lambda \in \C^\times$ and a logarithm $t \in \C$ such that $\exp(t) = \lambda$. Note that for $(E,\varphi)$ a $(G_0,\lieg_1)$-Higgs bundle, the element $\tilde{\xi}_\lambda := \exp(t\zeta) \in \Aut(E)$ defines an automorphism of $E$ since $G_0$ is abelian. Moreover, we have the identity $\tilde{\xi}_\lambda^*\varphi = \lambda^{-1}\varphi$. We see that $\tilde{\xi}_\lambda^{-1}$ is an isomorphism of $(G_0,\lieg_1)$-Higgs bundles $(E,\varphi) \xrightarrow{\simeq} (E,\lambda \varphi)$. This also extends to a $G$-Higgs bundle isomorphism which we will denote equally. 
	
	Let $(E',\varphi')$ be the Hecke transformation of $E$ by $\sigma$ and $(E'_\lambda,\varphi'_\lambda)$ the Hecke transformation by $\sigma_\lambda$. We want to show that the diagram
	\[\begin{tikzcd}[ampersand replacement=\&]
		{E'} \&\& E \\
		\\
		{E'_\lambda} \&\& E
		\arrow["\psi", from=1-1, to=1-3]
		\arrow["{f_\lambda}"', from=1-1, to=3-1]
		\arrow["{\tilde{\xi}_\lambda^{-1}}", from=1-3, to=3-3]
		\arrow["{\psi_\lambda}", from=3-1, to=3-3]
	\end{tikzcd},\]
	
	\noindent which is well defined over $C_0$, extends to an isomorphism $f_\lambda$ which sends $\lambda\varphi'$ to $\varphi'_\lambda$. For this, we inspect it locally. We push the diagram down the trivialisations to $C_0$ to find the local expression $f_\lambda^0$:
	
	\[\begin{tikzcd}[ampersand replacement=\&]
		{C_0 \times G} \&\& {C_0 \times G} \\
		\\
		{C_0 \times G} \&\& {C_0 \times G}
		\arrow["{\text{Id}}", from=1-1, to=1-3]
		\arrow["{f_\lambda^0}"', from=1-1, to=3-1]
		\arrow["{L_{\tilde{\xi}_\lambda^{-1}}}", from=1-3, to=3-3]
		\arrow["{\text{Id}}", from=3-1, to=3-3]
	\end{tikzcd}.\]
	
	We see that $f_\lambda^0 = L_{\tilde{\xi}_\lambda^{-1}}$ is given by left multiplication with $\tilde{\xi}_\lambda^{-1}$. Similarly we find $f_\lambda^1$:
	\[\begin{tikzcd}[ampersand replacement=\&]
		{C_1 \times G} \&\& {C_1 \times G} \\
		\\
		{C_1 \times G} \&\& {C_1 \times G}
		\arrow["{L_\sigma}", from=1-1, to=1-3]
		\arrow["{f_\lambda^1}"', from=1-1, to=3-1]
		\arrow["{L_{\tilde{\xi}_\lambda^{-1}}}", from=1-3, to=3-3]
		\arrow["{L_{\sigma_\lambda}}", from=3-1, to=3-3]
	\end{tikzcd},\]
	\noindent so that $f^1_\lambda = L_{\sigma_\lambda^{-1}\tilde{\xi}_\lambda^{-1}\sigma} = L_{\Id_G}$. Notice that our choice of $\sigma_{\lambda}$ ensures that $\sigma_\lambda^{-1}\tilde{\xi}_\lambda^{-1}\sigma \in L^+G$ so that the isomorphism is also well defined at $c$. We finally need to check that $\lambda\varphi'$ is sent to $\varphi'_{\lambda}$.
	
	On $C_0$ we have $(f_\lambda^0)^*(\varphi'_\lambda)^0 =  (f_\lambda^0)^*\varphi_0 = (L_{\tilde{\xi}_\lambda^{-1}})^*\varphi_0$. We find this pullback as follows:
	$$L_{\tilde{\xi}_\lambda^{-1}}^*\varphi_0(p,g) = \varphi_0(p,\tilde{\xi}_\lambda^{-1}g) = \Ad_{\tilde{\xi}_\lambda}\varphi_0(p,g) = \lambda\varphi_0(p,g).$$
	We used that $(E,\varphi)$ reduces to a $(G_0,\lieg_1)$-Higgs pair so we can assume that $g \in G_0$ and hence it commutes with $\tilde{\xi}_\lambda^{-1}$, and also $[\zeta,\varphi_0]=\lambda\varphi_0$. We do the same check on $C_1$:
	$$(f_\lambda^1)^*(\varphi_\lambda')^1 = (L_{\sigma_\lambda^{-1}\tilde{\xi}_\lambda^{-1}\sigma})^*\Ad_{\sigma_\lambda^{-1}}\varphi_1= (L_{\Id_G})^*\Ad_{\sigma^{-1}\tilde{\xi}_\lambda}\varphi_1 = \lambda \Ad_{\sigma^{-1}}\varphi_1.$$

	 Finally, since $\sigma_\lambda$ is in the same $L^+G$-orbit as $\sigma$, they belong to the same Schubert cell.
	\end{proof}
	
	\section{Classification of very stable $G$-Higgs bundles of Borel type}\label{sectheorem}
	
	In this section we classify, depending on the multiplicity vector, which fixed points of Borel type are very stable. We will need a series of preliminary results on Hecke transformations. For these, the base point $c \in C$ will be fixed throughout and omitted from the notation. 
	
	First, we consider Hecke transformations of such a point which produce another fixed point.
	
	\begin{proposition}\label{fixedhecke}
		Let $(E,\varphi)$ be a $\mathbb C^\times$-fixed point of Borel type with multiplicity vector $(a_1,\dots,a_r)$ and fixed trivialisation $\{(C_0,t_0),(C_1,t_1)\}$. Assume that the coweight $\mu := \sum_ib_i\omega_i^\vee \in \liet$ lifts to a cocharacter of $T \subseteq G$. Suppose also that $b_i \le a_i$ for $1 \le i \le r$. Then $z^\mu \in \mathcal H_{(E,\varphi)}$, and the Hecke transformation is another $\C^\times$-fixed point of Borel type with multiplicity vector $(a_i-b_i)_i$.
	\end{proposition}
	\begin{proof}
		By definition of the multiplicity vector, if $z$ is the local coordinate at $C_1$ we can write $\varphi_1 = \sum_iz^{a_i}u_i(z)X_i$, where $X_i \in \lieg_{\alpha_i}$ are simple root vectors and $u_i(z)$ are units in $\C[[z]]$. Now, write $z = \exp(t)$ and use the fact that $[\omega_i^\vee, X_j] = \alpha_j(\omega_i^\vee)X_j = \delta_{ij}X_j$ to compute:
		$$\Ad_{z^{k\omega^\vee_i}}(X_j) = \Ad_{\exp(tk\omega^\vee_i)}(X_j) = \exp(\ad(tk\omega^\vee_i))(X_j) = \begin{cases}
			z^kX_j & i=j,\\
			X_j & i \neq j
		\end{cases}.$$
		Then, we have
		$$\Ad_{z^{-\mu}}(\varphi_1) = \sum_iz^{a_i-b_i}u_i(z)X_i,$$
		\noindent as desired.
	\end{proof}
	
	In order to relate different fixed points using Hecke transformations, we need to consider curves in the space $\mathcal H_{(E,\varphi)}$ of Hecke transformations. Recall the $\C^\times$-action on $\Gr_G$ from Proposition \ref{cstargrass}, sending $\sigma$ to $\sigma_{\lambda} = \xi_{\lambda}^{-1}\sigma$.
	
	\begin{proposition}\label{heckecurve}
		Let $\mu \in X_*(T)$ be a cocharacter and $\alpha^\vee \in \Phi^\vee_+$ a positive coroot with corresponding positive root $\alpha \in \Phi^+$. Assume that $m:=\alpha(\mu) - 1 \ge 0$. There exists an element 
		$$\sigma_{\mu,\alpha} := Li_\alpha\begin{pmatrix}
			1 & z^m \\
			0 & 1
		\end{pmatrix}z^\mu \in \Gr_G$$
		\noindent such that $\lambda \mapsto (\sigma_{\mu,\alpha})_{\lambda}$ is contained in $\Gr_{\mu}$ and has limits $z^{\mu-\alpha^\vee}$ when $\lambda \to 0$ and $z^\mu$ when $\lambda \to \infty$.
	\end{proposition}
	
	\begin{proof}
		Since $\Ad_{\xi_\lambda}$ acts with weight $\heightr(\alpha)$ on $X_\alpha$, the adjoint action of $\xi_\lambda^{-1}$ maps $X_{\alpha} \mapsto \lambda^{-\heightr(\alpha)}X_{\alpha}$. This implies that we have the following curve:
		$$\lambda \mapsto \xi_\lambda^{-1}Li_\alpha\begin{pmatrix}
			1 & z^m \\
			0 & 1
		\end{pmatrix}z^\mu = \xi_\lambda^{-1}Li_\alpha\begin{pmatrix}
			1 & z^m \\
			0 & 1
		\end{pmatrix}z^\mu\xi_\lambda = Li_\alpha\begin{pmatrix}
			1 & \lambda^{-\heightr(\alpha)} z^m \\
			0 & 1
		\end{pmatrix}z^\mu.$$
		
		The result now follows directly from Lemma \ref{grassmanniancurve} and its proof, noticing that the curve given by $(\sigma_{\mu,\alpha})_{\lambda}$ coincides with $C_{\mu,\alpha} \setminus \{z^\mu, z^{\mu-\alpha}\}$.
	\end{proof} 
	
	We will also need to know when can we perform Hecke transformations by the curve in Proposition \ref{heckecurve}.
	
	\begin{lemma}\label{lemmahecke}
		Let $(E,\varphi)$ be a $\mathbb C^\times$-fixed point of Borel type with multiplicity vector $(a_1,\dots,a_r)$ and fixed trivialisation $\{(C_0,t_0),(C_1,t_1)\}$. Consider the coweight $\mu := \sum_ib_i\omega_i^\vee \in \liet$ and assume that it lifts to a cocharacter of $G$. Consider also the positive root $\alpha \in \Delta^+$.  	
		Then, $\sigma_{\mu,\alpha} \in \mathcal H_{(E,\varphi)}$ if and only if for all $i \in \{1,\dots,r\}$ and $l_i\ge0$ such that $\alpha_i + l_i\alpha \in \Phi$, we have
		$$a_i - b_i - l_i \ge 0.$$
	\end{lemma}
	
	\begin{proof}
		Set $m = \alpha(\mu)-1$ and $\nu := \exp(z^mX_\alpha)$ where $X_\alpha$ is a root vector for $\alpha$. We have to check that $\Ad_{z^{-\mu}\nu^{-1}}\varphi_1 \in L^+\lieg$. We have that $\varphi_1 = \sum_iz^{a_i}u_i(z)X_i$ with $u_i(z)$ units in $\C[[z]]$. Using the fact that, whenever $\alpha+\beta$ is a root, we have $[X_\alpha,X_\beta] \in \gen{X_{\alpha+\beta}}$, we get:
		$$\Ad_{\nu^{-1}}\varphi_1 = \sum_{l \ge 0}z^{lm}\varphi^{l\alpha}.$$
		Here
		$$\varphi^{l\alpha} := \sum_{i=1}^rD_{i,l}(z)z^{a_i}X_{\alpha_i + l\alpha},$$
		\noindent where the $D_{i,l}(z)$ are units in $\C[[z]]$ and with the convention that $X_{\alpha_i + l\alpha} = 0$ when $\alpha_i + l\alpha$ is not a root. Notice that the $\varphi^{l\alpha}$ are linearly independent from each other because the vectors appearing in them correspond to roots of height $\heightr(l\alpha)+1$. Hence $\sigma_{\mu,\alpha} \in \mathcal H_{(E,\varphi)}$ if and only if $\Ad_{z^{-\mu}}z^{lm}\varphi^{l\alpha} \in L^+\lieg$ for all $l\ge0$. Now, using that $\alpha(\mu) = m+1$, we get:
		$$\Ad_{z^{-\mu}} z^{lm}\varphi^{l\alpha} = \sum_{i=1}^rD_{i,l}(z)z^{lm}z^{a_i}z^{-b_i-l\alpha(\mu)}X_{\alpha_i+l\alpha} = \sum_{i=1}^rD_{i,l}(z)z^{a_i-b_i-l}X_{\alpha_i+l\alpha},$$
		\noindent and the lemma follows.
	\end{proof}
	
	 Recall that a dominant coweight $\mu \in X^+_*(T)$ is called \textit{minuscule} if it is minimal with respect to the dominance order.
	 
	\begin{corollary}\label{roothecke}
		Let $(E,\varphi)$ be a $\mathbb C^\times$-fixed point of Borel type with multiplicity vector $(a_1,\dots,a_r)$ and fixed trivialisation $\{(C_0,t_0),(C_1,t_1)\}$. Define the dominant coweight $\mu = \sum_ia_i\omega_i^\vee$. Suppose that $\mu$ is not minuscule. Then there exists a positive root $\alpha \in \Delta^+$ such that $\mu - \alpha^\vee$ is dominant and $\sigma_{\alpha^\vee,\alpha} \in \mathcal H_{(E,\varphi)}$.
	\end{corollary}
	
	\begin{proof}
		It suffices to prove this for every simple Dynkin type. The strategy is to give an explicit $\alpha$ (or its coroot) such that the hypotheses of Lemma \ref{lemmahecke} are verified for $\sigma_{\alpha^\vee,\alpha}$. We follow the numbering for the simple roots in \cite[Appendix C]{knapp}.
		\begin{itemize}
			\item Types $A_n$, $D_n$, $E_6$, $E_7$ and $E_8$: in simply-laced types, it can be proven by analysing the possible root string lengths that any $\alpha \in \Delta^+$ such that $\mu - \alpha^\vee$ is dominant (which exists since $\mu$ is not minuscule) verifies $\sigma_{\alpha^\vee, \alpha} \in \mathcal H_{(E,\varphi)}$ via Lemma \ref{lemmahecke}. Alternatively, the check becomes immediate in the particular choice of the positive root with the special properties in \cite[Theorem 2.8ab]{stembridge_partial_1998}.
			\item Type $B_n$: recall that the type of the Langlands dual is $C_n$. It suffices to prove it for $\mu \in \{2\omega_1^\vee,\omega_2^\vee,\dots,\omega_n^\vee\}$ since any other non-minuscule dominant coweight contains one of those as a summand. In the case $\mu = 2\omega_1^\vee$, we may take $\alpha^\vee = 2\omega_1^\vee-\omega_2^\vee$. If $\mu = \omega_2^\vee$ we can choose $\alpha^\vee = \omega_2^\vee$. For $\mu = \omega_k^\vee$ with $2 < k \le n$, it suffices to take $\alpha^\vee = \omega_k^\vee - \omega_{k-2}^\vee$. 
			\item Type $C_n$: the type of the Langlands dual is $B_n$. As before, we may assume $\mu \in \{\omega_1^\vee,\dots,\omega_{n-1}^\vee, 2\omega_n^\vee\}$. If $\mu = \omega_1^\vee$, we take $\alpha^\vee = \omega_1^\vee$. If $\mu = \omega_k^\vee$ with $2 \le k < n$, we take $\alpha^\vee = \omega_k^\vee - \omega_{k-1}^\vee$. Finally, if $\mu = 2\omega_n^\vee$ we take $\alpha^\vee = 2\omega_n^\vee - \omega_{n-1}^\vee$.
			\item Type $F_4$: we may assume that $\mu$ is fundamental. For $\mu = \omega_1^\vee$ we may take $\alpha^\vee = \omega_1^\vee - \omega_4^\vee$. For $\mu = \omega_2^\vee$ we use $\alpha^\vee = \omega_2^\vee - \omega_1^\vee$. The choice for $\mu = \omega_3^\vee$ is $\alpha^\vee = \omega_3^\vee - \omega_4^\vee$, and for $\mu = \omega_4^\vee$ it is $\alpha^\vee = \omega_4^\vee$.
			\item Type $G_2$: again, it suffices to do the case where $\mu$ is fundamental. For $\mu = \omega_1^\vee$ we select $\alpha^\vee = \omega_1^\vee - \omega_2^\vee$ and for $\mu = \omega_2^\vee$ we choose $\alpha^\vee = \omega_2^\vee$.
		\end{itemize} 
	\end{proof}
	As a last ingredient, we need to check the stability of the resulting fixed point after some special kinds of Hecke transformations.
	
	\begin{proposition}\label{stabilitymain}
		Let $(E,\varphi)$ be a stable $\mathbb C^\times$-fixed point of Borel type with fixed trivialisation $\{(C_0,t_0),(C_1,t_1)\}$.  Consider $\mu \in \mathfrak t$ coming from a cocharacter of $T \subseteq G$ and such that $\chi_+(\mu) \ge 0$ for any dominant character $\chi_+ \in X_+^*(G)$ (i.e. $\mu$ is a linear combination of simple coroots with non-negative coefficients). Assume that $z^\mu \in \mathcal H_{(E,\varphi)}$. Then $(E',\varphi') := \mathcal H_{z^\mu}(E,\varphi)$ is also stable.
	\end{proposition}
	
	\begin{proof}
		Let $s \in i \liek$. After conjugation we may assume that $P_s = P_S \subseteq G$, where $P_S$ is the standard parabolic subgroup corresponding to a subset of simple roots $S \subseteq \Pi$, and the character $\chi_s$ is of the form $\chi_s = \sum_{\alpha \in S}n_\alpha \omega_{\alpha}$ where $n_\alpha \le 0$ and $\omega_\alpha$ is the fundamental weight corresponding to $\alpha$. Suppose given a reduction of structure group $\sigma' \in H^0(E'(G/P_S))$ such that $\varphi' \in H^0(E'_{\sigma'}(\lieg_{s}) \otimes K_C)$. We have to show that $\deg(E')(\sigma',s) > 0$.
		
		We will see that $\sigma'$ induces a reduction $\sigma \in H^0(E(G/P_S))$. Indeed, following the usual notation for the trivialisations, let $\sigma'_0$ and $\sigma'_1$ be the local sections, i.e. equivariant maps $\sigma'_i : C_0 \times G \to G/P_S$ obtained by pulling back $\sigma'$ via the trivialisation. As usual, we set
		$$\sigma_0 := \sigma'_0.$$
		Now consider the map $\tilde{\sigma}_1 : C_{01} \to G/P_S$ defined by $\tilde{\sigma}_1(p) := \mu(p) \sigma'_1(p,1_G)$ (as always, we are viewing $p$ as a point in the punctured disk $C_{01}$ and hence it makes sense to evaluate $\mu$ on it). Since $P_S$ is parabolic, $G/P_S$ is projective and thus, by the valuative criterion of properness, this map extends uniquely to $\bar{\sigma}_1 : C_1 \to G/P_S$. We define $\sigma_1 : C_1 \times G \to G/P_S$ equivariantly, as
		$$\sigma_1(p,g) := g^{-1}\bar{\sigma}_1(p).$$
		These maps glue to a global section $\sigma$ since they agree on the intersection $C_{01}$. Indeed,
		$$\sigma_1(p,g) = g^{-1}\mu\sigma_1'(p,1_G) =  \sigma_1'(p, \mu^{-1} g) = \sigma_0'(p,f_E'\mu^{-1}g) = \sigma'_0(p,f_E g) = \sigma_0(p,f_Eg).$$
		It is also clear from the construction that $\varphi' \in H^0(E'_{\sigma'}(\lieg_{s}) \otimes K_C)$ implies $\varphi \in H^0(E_\sigma(\lieg_{s}) \otimes K_C)$. Thus, stability of $(E,\varphi)$ gives $\deg E(\sigma,s) > 0$.
		
		Choose $N \ge 1$ such that the multiples $N\omega_\alpha$ lift to characters $\chi_\alpha$ of the group. Note that by assumption $\chi_\alpha(\mu) \ge 0$. Recall that we can compute
		$$\deg E(\sigma,s) = \frac{1}{N}\sum_{\alpha \in S}n_\alpha \deg(\chi_{\alpha}(E_\sigma)),$$
		\noindent where $E_{\sigma}$ is the reduced bundle. As $G_0 \subseteq P_S$, we may assume that the transition functions for the reduced bundle still verify $f_{E_{\sigma'}'} = f_{E_\sigma}\mu$. Hence 
		$$\chi_{\alpha}(E'_{\sigma'}) = \chi_\alpha(E_\sigma) \otimes \mathcal O(-\chi_\alpha(\mu)c),$$
		\noindent so that $$\deg(\chi_{\alpha}(E'_{\sigma'})) = \deg(\chi_\alpha(E_\sigma)) - \chi_\alpha(\mu) \le \deg(\chi_\alpha(E_\sigma)).$$
		Since the coefficients $n_\alpha$ are not positive, we conclude that
		$$\deg E'(\sigma',s) \ge \deg E(\sigma,s) > 0,$$
		\noindent as desired.
	\end{proof}
	
	Simple coroots and fundamental coweights fall under the hypotheses of the proposition, as they can be expressed as a non-negative combination of the simple coroots. This gives the following corollaries.
	
	\begin{corollary}\label{stability}
		Let $(E,\varphi)$ be a stable $\C^\times$-fixed point of Borel type with fixed suitable trivialisation $\{(C_0,t_0),(C_1,t_1)\}$ and let $\alpha \in \Delta^+$ be such that $z^{\alpha^\vee} \in \mathcal H_{(E,\varphi)}$. Then $(E',\varphi') := \mathcal H_{z^{\alpha^\vee}}(E,\varphi)$ is also stable.
	\end{corollary}
	
	\begin{proof}
		Suppose that the positive multiple of a fundamental weight $N\omega_j$ lifts to a character $\chi$. Then by the duality of simple coroots and fundamental weights, we get that $\chi(\alpha^\vee_i) = N\delta_{ij} \ge 0$, so Proposition \ref{stabilitymain} applies.
	\end{proof}
	
	\begin{corollary}\label{stabilityweight}
		Suppose that $\omega_i^\vee$ lifts to a cocharacter of $G$. Let $(E,\varphi)$ be a stable $\C^\times$-fixed point of Borel type with a suitable (as in Section \ref{heckeprin}) fixed trivialisation $\{(C_0,t_0),(C_1,t_1)\}$ such that $z^{\omega_i^\vee} \in \mathcal H_{(E,\varphi)}$. Then $(E',\varphi') := \mathcal H_{z^{\omega_i^\vee}}(E,\varphi)$ is also stable.
	\end{corollary}
	
	\begin{proof}
		Suppose that the positive multiple of a fundamental weight $N\omega_j$ lifts to a character $\chi$. Let $N_{ik}$ be change of basis (given by the inverse of the Cartan matrix for the dual) so that $\omega_i^\vee = \sum_{k}N_{ik}\alpha_{k}^\vee$. It is a fact that $N_{ik} \ge 0$ for all types of Cartan matrix. Now, $\chi(\omega_i^\vee) = N\cdot N_{ij} \ge 0$, so Proposition \ref{stabilitymain} applies.
	\end{proof}
	
	Recall from Proposition \ref{smoothness} that stability alone does not guarantee smoothness in $\mathcal M(G)$ which is required for the study of the Bia\l ynicki-Birula decomposition using the tangent space. Due to this, we need the following result.
	
	\begin{proposition}\label{simplicity}
		Let $(E,\varphi)$ be a simple $\C^\times$-fixed point of Borel with a suitable (as in Section \ref{heckeprin}) fixed trivialisation $\{(C_0,t_0),(C_1,t_1)\}$ and $\mu \in X_*^+(T)$ a cocharacter with $z^\mu \in \mathcal H_{(E,\varphi)}$. Assume that $(E',\varphi') := \mathcal H_{z^\mu}(E,\varphi)$ is stable. Then it is also simple.
	\end{proposition}
	
	\begin{proof}
		Let $f' \in \Aut(E',\varphi')$. We regard $f'$ as a global section of the adjoint bundle $E'(G) := E' \times_{\conj} G$ where $\conj : G \to \Aut(G)$ is given by left conjugation. Compatibility with the Higgs field is given by $\Ad_f\varphi = \varphi$. Pulling back to the trivialisation we get the local expressions
		$$f'_i : C_i \times G \to G,$$
		\noindent for $i \in \{0,1\}$, which are $G$-equivariant in the sense that $f'_i(p,gh) = h^{-1}f'_i(p,g)h$.
		
		The stability of $(E',\varphi')$ implies \cite[Proposition 3.15]{garcia-prada_connectedness_2017} that $f'_i(p,g)$ is a semisimple element in $G$. Since $G$ is connected, there is a maximal torus containing the semisimple element $f_1'(c,1_G)$ and, since all maximal tori are conjugate, we may assume (applying a suitable left multiplication after the trivialisations) that $f_1'(c,1_G) \in T = G_0$.
		
		Now set
		$$f_0(p,g) := f_0'(p,g),$$
		$$f_1(p,g) := f'_1(p,\mu^{-1}(p)g).$$ 
		
		The expression for $f_1$ is well defined also at $c$, because we have $f_1(c,1_G) = f'_1 (c,\mu^{-1}) = \mu f'_1(c,1_G) \mu^{-1} = f'_1(c,1_G)$ using that $f'_1(c,1_G) \in G_0$, that $\mu$ is a loop in $G_0 = T$ and that $G_0$ is abelian. Then, for arbitrary $g \in G$, we have $f_1(c,g) = g^{-1}f_1(c,1_G)g$.
		
		These two expressions glue in the intersection $C_{01}$ to give a global automorphism $f \in \Aut(E)$, since
		$$f_0(p,f_E g) = f_0'(p,f_Eg) = f_0'(p, f_E'\mu^{-1}g) = f_1'(p,\mu^{-1}g) = f_1(p,g).$$
		
		Moreover, $f$ preserves the Higgs field, as can be checked locally:
		$$\Ad_{f_0}\varphi_0 = \Ad_{f_0'}\varphi_0' = \varphi_0' = \varphi_0,$$
		\noindent and
		$$\Ad_{f_1}\varphi_1 = \Ad_{\mu f_1' \mu^{-1}}\Ad_{\mu}\varphi_1' = \Ad_{\mu f_1'}\varphi_1' = \Ad_{\mu} \varphi_1' = \varphi_1.$$
		
		Thus, by simplicity of $(E,\varphi)$, we get that $f \in Z(G)$ is a constant in the centre, and hence $f' \in Z(G)$ as well. 
	\end{proof}
	
	We are finally in position of proving the main result.
	
	\begin{theorem}\label{verystablecharact}
		Let $(E,\varphi)$ be a smooth $\C^\times$-fixed point of Borel type with multiplicity vector $(a^c_1,\dots,a^c_r)$ at $c \in C$. Set $\mu_c = \sum_ia^c_i\omega_i^\vee$. Then $(E,\varphi)$ is very stable if and only if $\mu_c$ is minuscule for all $c \in C$.
	\end{theorem}
	\begin{proof}
		Suppose that some $\mu_c =: \mu$ is not minuscule. We can use Corollary \ref{roothecke} to find a coroot $\alpha^\vee$ with $\mu - \alpha^\vee$ dominant and such that $\sigma_{\alpha^\vee,\alpha} \in \mathcal H^c_{(E,\varphi)}$. Then, by Proposition \ref{heckecurve} there is a Hecke curve $C_{\alpha}$ in $\mathcal H^c_{(E,\varphi)}$ compatible with the $\C^\times$-action connecting $0$ at $\lambda \to 0$ and $\alpha^\vee$ at $\lambda \to \infty$. The corresponding Hecke curve in Higgs bundles, by Proposition \ref{fixedhecke}, connects $(E,\varphi)$ at $\lambda \to 0$ with $(E',\varphi')$ at $\lambda \to \infty$ which is a fixed point with multiplicities given by the non-negative coordinates of $\mu - \alpha^\vee$ in the basis of fundamental coweights. By Corollary \ref{stability} $(E',\varphi')$ is stable. By openness of stability, the whole curve is inside of the moduli space, and the result follows.
		
		For the converse we start by assuming that $G$ is of adjoint type (i.e. $G = \Ad(G)$). In this case the fundamental coweights $\omega_i^\vee \in \liet$ lift to cocharacters of $G$. We proceed by induction on $N := \sum_{c \in C}\sum_{i=1}^ra_i^c$. The case $N=0$ corresponds to Example \ref{uniformising} which was seen to be very stable as a consequence of the upward flow being a section of the Hitchin map. Now suppose that there is some nonzero $\mu_c =: \mu$. It is minuscule, in particular of the form $\mu = \omega_i^\vee$ (although not every choice of $i$ is minuscule). Now assume that $(E',\varphi')$ is a nilpotent Higgs bundle on the upward flow of $(E,\varphi)$. By Proposition \ref{upwardflow}, we can choose a trivialisation $\{(C_0',t_0'), (C_1',t_1')\}$ of $E'$ with the same notation as in Section \ref{heckeprin} such that
		$$\varphi'_1(z) = zu_i(z)X_{\alpha_i} + \sum_{j \neq i}u_j(z)X_{\alpha_j} + \sum_{\alpha \in \Delta^{-}}f_\alpha(z)X_{\alpha} + \sum_{j=1}^rf_j(z)\alpha^\vee_i,$$
		\noindent the first two summands corresponding to $\varphi_1$ and the rest to $\lieg_{\le0}$, and the $u_j(z) \in \C[[z]]$ being invertible. Thus
		$$\varphi''_1(z) := \Ad_{(\omega_i^\vee)^{-1}}\varphi'_1(z) = u_i(z)X_{\alpha_i} + \sum_{j \neq i}u_j(z)X_{\alpha_j} + \sum_{\alpha \in \Delta^{-}}z^{-\omega_i^\vee(\alpha)}f_\alpha(z)X_{\alpha} + \sum_{j=1}^rf_j(z)\alpha^\vee_i,$$
		\noindent where $\omega_i^\vee(\alpha) \le 0$ since the $\alpha$ are negative. Thus, $\omega_i^\vee \in \mathcal H^c_{(E',\varphi')}$. This yields $(E'',\varphi'') := \mathcal H_{z^{\omega_i^\vee}}^c(E',\varphi')$. Since $z^{\omega_i^\vee}$ is a loop with values in $G_0 \subseteq B^{opp}$, $(E'',\varphi'')$ still reduces to $(B^{opp}, \lieg_{\le 1})$. Moreover, using the facts that $\Gr(E',\varphi') \simeq (E,\varphi)$, that the transition function satisfies $f_{E''} = f_{E'}z^{\omega_i^\vee}$ and the expression of $\varphi_1(z)''$ from above, it follows that 
		$$\Gr(E'',\varphi'') \simeq \mathcal H_{z^{\omega_i^\vee}}^c(E,\varphi) =: (E''',\varphi'''),$$
		\noindent where the Hecke transformation is taken with respect to the same trivialisation. From Proposition \ref{fixedhecke} it follows that $\varphi'''_1 = \sum_{j=1}^ru_j(z)X_{\alpha_j}$ and hence $(E''',\varphi''')$ still satisfies the condition (note that by Corollary \ref{stabilityweight} it is still stable and by Proposition \ref{simplicity} it is simple) 
		but with the value of $N$ decreased by one, so by the induction hypothesis it is very stable. Moreover, $(E'',\varphi'')$ is nilpotent because the Hecke transformation preserves the Hitchin fibres, and the previous discussion together with Proposition \ref{upwardflow} shows that it is in the upward flow of $(E''',\varphi''')$. Hence
		$$(E'',\varphi'') \simeq (E''',\varphi''').$$
		Note, however, that the previous isomorphism need not correspond the induced trivialisations with each other and hence it need not fix the elements in the Hecke transformation spaces. In other words, even if we know that with our working trivialisations we have $\mathcal H^c_{z^{-\omega_i^\vee}}(E'',\varphi'') \simeq (E',\varphi')$ and $\mathcal H^c_{z^{-\omega_i^\vee}}(E''',\varphi''') \simeq (E,\varphi)$, a priori we cannot guarantee that both are isomorphic because the isomorphism between $(E'',\varphi'')$ and $(E''',\varphi''')$ need not fix $z^{-\omega_i^\vee}$ as transition function. However, even if the trivialisation changes via the isomorphism, the (Schubert) type of Hecke transformation is well-defined. Hence, we know that there is some Hecke transformation $\sigma \in \mathcal H^c_{(E''',\varphi''')}$ of type $z^{-\omega_i^\vee}$ such that in our working trivialisations we have
		$$\mathcal H_\sigma^c(E''',\varphi''') \simeq \mathcal H_{z^{-\omega_i^\vee}}^c(E'',\varphi'') \simeq (E',\varphi').$$
		We conclude by noticing from Remark \ref{minusculekostantsection} that, since $\omega_i^\vee$ is minuscule and 
		$$\varphi'''_1(c) = \sum_{j=1}^ru_j(0)X_{\alpha_j} \in \lieg$$ 
		\noindent is regular nilpotent, the Hecke transformation space of the given type for $(E''',\varphi''')$ consists of a single element. Thus, $\sigma = z^{-\omega_i^\vee}$ and we conclude
		$$(E',\varphi') = \mathcal H_{z^{-\omega_i^\vee}}^c(E'',\varphi'') \simeq \mathcal H_{z^{-\omega_i^\vee}}^c(E''',\varphi''') = (E,\varphi),$$
		\noindent as desired.
		
		Finally, if $G$ is not of adjoint type we may use the map $\mathcal M(G) \to \mathcal M(\Ad(G))$ given by extension of structure group via the projection $G \to \Ad(G)$. If $(E,\varphi)$ were wobbly, its image in $\mathcal M(\Ad(G))$ would also be wobbly with the same multiplicities, a contradiction.
	\end{proof}
	
	We now illustrate the classification theorem for classical simple groups. We follow the numbering of the simple roots from the tables in \cite[Appendix C]{knapp}.
	
	\begin{example}
		If $G=\SL_n(\C)$ or $G=\PGL_n(\C)$, recall from Example \ref{typea} that a fixed point of Borel type consists of $n$ line bundles $(L_1,\dots,L_n)$ over $C$ (with trivial product in the case of $\SL_n(\C)$, and up to tensoring by the same line bundle for $\PGL_n(\C)$) and $n-1$ nonzero sections $\delta_i \in H^0(C,L_i^{*}L_{i+1}K_C)$. Since the minuscule coweights for type $A_{n-1}$ are $\{0,\omega_1^\vee,\dots,\omega_{n-1}^\vee\}$, a stable fixed point of Borel type is very stable if and only if the divisor $\delta_1 + \dots + \delta_{n-1}$ is reduced, recovering \cite[Theorem 4.16]{hausel_very_2022}.
	\end{example}
	
	\begin{example}
		For $G = \SO_{2n+1}(\C)$, recall from Example \ref{typeb} that fixed points of Borel type are given by $n$ line bundles $(L_1,\dots,L_n)$ over $C$, $n-1$ nonzero sections $\delta_i \in H^0(C, L_i^*L_{i+1}K_C)$ and a nonzero section $\eta \in H^0(C, L_n^*K_C)$. Minuscule coweights in type $B_n$ (corresponding to minuscule weights in type $C_n$) are $\{0, \omega_1^\vee\}$, so a smooth fixed point of Borel type is very stable if and only if the divisor $\delta_1+2\delta_2+\dots+2\delta_{n-1}+2\eta$ is reduced.
	\end{example}
	
	\begin{example}
		In the case of $G=\Sp_{2n}(\C)$, fixed points of Borel type are given by $n$ line bundles $(L_1,\dots,L_n)$ over $C$, $n-1$ nonzero sections $\delta_i \in H^0(C, L_i^*L_{i+1}K_C)$ and a nonzero section $\eta \in H^0(C,(L_n^2)^*K_C)$, as was seen in Example \ref{typec}. Therefore, since minuscule coweights in type $C_n$ are $\{0, \omega_n^\vee\}$, we have that a smooth fixed point of Borel type is very stable if and only if the divisor $2\delta_1+\dots+2\delta_{n-1}+\eta$ is reduced.
	\end{example}
	
	\begin{example}
		For $G = \SO_{2n}(\C)$, Example \ref{typed} shows that fixed points of Borel type are given by $n$ line bundles $(L_1,\dots,L_n)$ over $C$, $n-1$ nonzero sections $\delta_i \in H^0(C, L_i^*L_{i+1}K_C)$ and a nonzero section $\eta \in H^0(C, L_{n-1}^*L_n^*K_C)$. Minuscule coweights in type $D_n$ are $\{0,\omega_1^\vee, \omega_{n-1}^\vee, \omega_n^\vee\}$, so that a smooth fixed point of Borel type is very stable if and only if the divisor $\delta_1 + 2\delta_2+\dots+2\delta_{n-2}+\delta_{n-1} + \eta$ is reduced.
	\end{example}

	\begin{remark}
		We can use Remark \ref{toptype} to get restrictions on which fixed point components of Borel type can contain very stable points. Indeed, let $\mu_{can} := (2g-2)\sum_{i=1}^r\omega_i^\vee$ be the coweight corresponding to the topological type in $\pi_1(G_0) \simeq X^+_*(T)$ of the canonical uniformising Higgs bundle. For a fixed point $(E,\varphi)$ of Borel type, let $\mu_c = \sum_{i=1}^ra_i^c\omega^\vee_i$ be as in Theorem \ref{verystablecharact} and $\mu := \sum_{c \in C}\mu_c$. Then, the topological type of $E$ as a $G_0$-bundle is $\mu^{-1}\mu_{can} \in \pi_1(G_0)$. This, together with Theorem \ref{verystablecharact}, implies that if a fixed point component contains very stable points then the corresponding topological type $\tau \in \pi_1(G_0) \simeq X^+_*(T)$ must verify that $\mu_{can}\tau^{-1}$ is a cocharacter corresponding to a sum of minuscule coweights in $\liet$. In the case of $G=\PGL_n(\C)$, this imposes no restriction as all fundamental coweights are minuscule (indeed, in this case there are very stable points in every component, see \cite[Corollary 4.19]{hausel_very_2022}). In any other type, however, this restricts which components are possible. In particular, in types without nonzero minuscule coweights ($E_8, F_4, G_2$) only components containing everywhere regular Higgs bundles may contain very stable points.
	\end{remark}
	
	\section{Virtual equivariant multiplicities}\label{secmult}
	
	We start by recalling the following notion from \cite[Definition 5.3]{hausel_very_2022}.
	
	\begin{definition}
		Let $\mathcal E := (E,\varphi) \in \mathcal M(G)^{s\C^\times}$ be a smooth fixed point of the $\C^\times$-action. We define its \textbf{virtual equivariant multiplicity} as
		$$m_{\mathcal E}(t) := \frac{\chi(\Sym (T^{+*}_\mathcal E))}{\chi(\Sym (\mathcal A_G^*))} \in \Z((t)),$$
		\noindent where $T^+_{\mathcal E} := T^+_{\mathcal E}\mathcal M(G)$ is the subspace of positive weights (see Section \ref{secbb}) of the tangent $T_{\mathcal E}\mathcal M(G)$, $\mathcal A_G$ is the Hitchin base and $\chi$ denotes the character of a $\C^\times$-representation, that is, the Laurent series where the coefficient of $t^{-k}$ is the dimension of the $k$-th weight subspace.
	\end{definition}
	
	Virtual equivariant multiplicities are interesting objects to associate to a fixed point due to the following properties.
	
	\begin{proposition}[{\cite[Theorem 5.2, Corollary 5.4]{hausel_very_2022}}]
		Let $\mathcal E \in \mathcal M(G)^{s\C^\times}$ be a very stable smooth fixed point. Then, $m_{\mathcal E}(t)$ is a monic, palyndromic polynomial with non-negative integer coefficients such that $m_{\mathcal E}(1)$ equals the multiplicity of the fixed point component containing $\mathcal E$ in the nilpotent cone $h^{-1}(0)$.
	\end{proposition}
	
	The goal of this section is to compute $m_{\mathcal E}(t)$ for the fixed points of Borel type in terms of the Lie-theoretic data of $\lieg$ and then show that in the very stable cases they agree with the Dynkin polynomials of the corresponding highest weight representation of the Langlands dual group $G^\vee$. This computation was outlined in \cite[Section 8.1]{hausel_very_2022} and we now show how it can be carried out within our framework.
	
	For a $\C^\times$-representation with positive weights $V = \bigoplus_{\lambda > 0}V_\lambda$, we have
	$$\chi(\Sym(V^*)) = \prod_{\lambda > 0}\frac{1}{(1-t^\lambda)^{\dim V_\lambda}}.$$ It now suffices to compute the dimensions of the weight spaces in both $T^+_{\mathcal E}\mathcal M(G)$ and $\mathcal A_G$.
	
	We can start by computing the denominator $\chi(\Sym (\mathcal A^*_G))$. Recall that if $\{e_j\}_{j=1}^r$ are the exponents of $\lieg$, then $\mathcal A_G = \bigoplus_{j=1}^rH^0(X, K_C^{e_j+1})$ and $\C^\times$ acts with weight $e_j+1$ on the $j$-th summand. Thus
	\begin{align}\label{denominator}
		\chi(\Sym (\mathcal A^*_G)) = \prod_{\lambda > 0}\frac{1}{(1-t^{e_j+1})^{(2e_j+1)(g-1)}}.
	\end{align}
	
	Now we compute $\chi(\Sym (T^{+*}_\mathcal E))$. Recall from the proof of Proposition \ref{upwardflow} that the subspace of $T_{\mathcal E}$ with weight $-j$ is $\mathbb H^1(E(\lieg_j) \xrightarrow{[\varphi,-]} E(\lieg_{j+1}) \otimes K_C)$, so it will be necessary to compute the Euler characteristics $\chi(E(\lieg_j))$ and, in turn, $\deg E(\lieg_j)$ for the different $\lieg_j$.
	
	The degree of a bundle associated to a representation $\rho : G_0 \to \GL(\lieg_j)$ can be computed by means of the character $\chi_\rho : G_0 \to \C^\times$ given by $\chi_\rho(g_0) := \det(\rho(g_0))$. From the structure of the Borel grading, we have:
	$$\chi_\rho(g_0) = \prod_{\alpha \in \Delta, \heightr(\alpha) = j}\alpha(g_0).$$ 
	
	We define as in Theorem \ref{verystablecharact} the multiplicity coweight $\mu := \sum_{c \in C} \sum_{j=1}^ra_j^c\omega_j^\vee$, where $(a_1^c,\dots,a_r^c)$ is the multiplicity vector of $\mathcal E$ at $c \in C$. Then, by Remark \ref{toptype}, we have that the topological type of $E$ in $\pi_1(G_0)$ is given by $\mu^{-1}\cdot\mu_{can}$ where $\mu_{can} = (2g-2)\sum_{i=1}^r\omega_i^\vee$ is the topological type of the canonical uniformising Higgs bundle (this can be deduced from the construction of the canonical uniformising via the principal $\liesl_2$-triple, as we have that $\frac{h}{2} = \sum_{j=1}^r\omega_j^\vee$ since $\alpha_j(\zeta) = 1$ for all $j$). Hence,
	
	$$d_{-j}(\mu) := \deg E(\lieg_j) = \sum_{\alpha \in \Delta, \heightr(\alpha) = j}-\alpha(\mu')+\alpha(\mu_{can}) = jr_{-j}(2g-2) + \sum_{\alpha \in \Delta, \heightr(\alpha) = -j}\alpha(\mu'),$$
	
	\noindent where
	$$r_{-j} := \dim \lieg_j = \left|\{\alpha \in \Delta : \heightr(\alpha) = j\}\right|.$$
	
	We used that for $\alpha$ a root of height $-j$ we have $-\alpha(\mu_{can}) = j(2g-2)$. We also define $d_0(\mu) = 0$ since $\deg E(\lieg_0) = 0$. By Riemann--Roch we deduce
	$$\chi(E(\lieg_j)) = d_{-j}(\mu) + r_{-j}(1-g),$$
	$$\chi(E(\lieg_j) \otimes K_C) = d_{-j}(\mu) + r_{-j}(g-1).$$
	
	Hence for $j>0$ we have:
	\begin{align}\label{numerator}
		\dim T^{+}_{\mathcal E, j} = \chi(E(\lieg_{-j+1}) \otimes K_C) - \chi(E(\lieg_{-j})) = d_{j-1}(\mu) - d_j(\mu) + (r_{j-1} + r_j)(g-1).
	\end{align}
	
	Using Equations \ref{denominator} and \ref{numerator}, we can compute any desired virtual equivariant multiplicities. Table \ref{tablemults} records the cases where $\mu$ is minuscule and nonzero. In other words, these are the cases where $\mathcal E$ only has a nonzero multiplicity vector at a single point $c \in C$ and said nonzero vector corresponds to a minuscule coweight. The numbering for the fundamental coweights follows the tables in \cite[Appendix C]{knapp}.
	
	\begin{table}[]\label{tablemults}
		\caption{Equivariant multiplicities for minuscule upward flows}
		\begin{tabular}{|l|l|l|}\hline
			Type of $\lieg$ & $\mu$                              & $m_{\mathcal E}(t)$                                                                                                         \\\hline\hline
			$A_n$           & $\omega_i^\vee$                    & $\prod_{j=1}^i\frac{1-t^{n-j+1}}{1-t^j}$                                                                                     \\\hline
			$B_n$           & $\omega_1^\vee$                    & $1+t+\dots+t^{2n-1}$                                                                                                         \\\hline
			$C_n$           & $\omega_n^\vee$                    & $\prod_{j=1}^n(1+t^j)$                                                                                                       \\\hline
			$D_n$           & $\omega_1^\vee$                    & $(1+t^{n-1})(1+t+\dots+t^{n-1})$                                                                                             \\
			& $\omega_{n-1}^\vee, \omega_n^\vee$ & $\prod_{j=1}^n(1+t^j)$                                                                                                       \\\hline
			$E_6$           & $\omega_1^\vee, \omega_6^\vee$     & $(1+t^4+t^8)(1+t+\dots+t^8)$                                                                 \\\hline
			$E_7$           & $\omega_7^\vee$                    & $(1+t^5)(1+t^9)(1+t+\dots+t^{13})$\\\hline
		\end{tabular}
	\end{table}
	
	A direct comparison with \cite[Table 1]{panyushev_weight_2004} shows the agreement of the virtual equivariant multiplicity in the above cases with the \textit{Dynkin polynomial} associated to a dominant coweight $\mu \in X^+_*(G)$, which is given by
	$$\mathcal D_\mu(t) := \prod_{\alpha \in \Delta^+}\frac{1-t^{\alpha(\rho^\vee + \mu)}}{1-t^{\alpha(\rho^\vee)}},$$
	where $\rho^\vee := \frac{1}{2}\sum_{\alpha \in \Delta^+}\alpha^\vee$.
	
	Thus, we can state the following proposition (see also \cite[Section 8.1]{hausel_very_2022}) computing the virtual equivariant multiplicities for very stable fixed points.
	
	\begin{proposition}
		Let $\omega_i^\vee \in X^+_*(G)$ be a minuscule dominant cocharacter, $c \in C$ be a point and $\mathcal E \in \mathcal M(G)^{\C^\times}$ be a fixed point such that its multiplicity vector at $c \in C$ is $a^c_j = \delta_{ij}$ and zero at any point other than $c$. Then, we have the identity
		$$m_{\mathcal E}(t) = \mathcal D_{\omega_i^\vee}(t)$$
		\noindent between its virtual equivariant multiplicity and the corresponding Dynkin polynomial.
	\end{proposition}
	
	\begin{remark}\label{virtmulthsr}
		Computing the virtual equivariant multiplicities for sums of non-minuscule fundamental coweights results in rational functions which are not polynomials, providing another indirect proof of one of the directions in Theorem \ref{verystablecharact}. This is not always the case: in other fixed point types (e.g. \cite{peon-nieto_wobbly_2023}), virtual equivariant multiplicities are often polynomials even in components without very stable points. 
	\end{remark}
	
	\bibliographystyle{plain}
	\bibliography{bibliography}
	
\end{document}